\documentclass[12pt]{amsart}
\usepackage{amstext, amsmath, amssymb, amsfonts, latexsym}
\usepackage{graphicx}
\usepackage{amsthm}
\usepackage{hyperref}
\usepackage{xcolor}
\usepackage{tikz}

\usepackage[left=3cm, right=3cm, bottom=4cm]{geometry} 
\numberwithin{figure}{section}

\newtheorem{maintheorem}{Theorem}
\newtheorem{maintheoremb}{Theorem}

\newtheorem{theorem}{Theorem}[section]

\newtheorem{lemma}[theorem]{Lemma}
\newtheorem{corollary}{Corollary}

\newtheorem{example}[theorem]{Example}

\newcommand{\dy}{\displaystyle}
\newcommand{\e}{\ensuremath{\varepsilon}}
\newcommand{\F}{\ensuremath{\mathcal{F}}}
\newcommand{\A}{\ensuremath{\mathcal{A}}}

\newcommand{\Ha}{\ensuremath{\mathcal{H}}}

\newcommand{\N}{\ensuremath{\mathbb{N}}}
\newcommand{\R}{\ensuremath{\mathbb{R}}}

\newcommand{\Pol}{\ensuremath{\mathbb{P}}}
\newcommand{\Exp}{\ensuremath{\mathbb{E}}}
\newcommand{\D}{\ensuremath{\mathbb{D}}}

\newcommand{\OR}{\ensuremath{\mathcal{O}}}
\newcommand{\OG}{\ensuremath{\mathbb{O}}}
\newcommand{\COG}{\ensuremath{\overline{\mathbb{O}}}}

\title{Polynomial entropy of Morse-Smale diffeomorphisms on surfaces}

\author{Javier Correa}
\address{Universidade Federal de Minas Gerais}
\email {jcorrea@mat.ufmg.br}

\author{Hellen de Paula} 
\address{Centro Federal de Educação Tecnológica de Minas Gerais} 
\email {hellenlimadepaula@ufmg.br}

\thanks{The second author has been supported by CAPES}

\begin{document}

\begin{abstract}
A classical problem in dynamical systems is to measure the complexity of a map in terms of its orbits, and one of the main concepts used to achieve this goal is entropy. Nonetheless, many interesting families of dynamical systems have every element with zero-entropy. One of these are Morse-Smale diffeomorphisms. In this work, we compute the generalized entropy of Morse-Smale diffeomorphisms on surfaces, based on which we deduce their polynomial entropy. We also apply our technique to compute the dispersion of the orbits of maps on the border of chaos with mild dissipation. 
\end{abstract}

 \maketitle

\vspace{2cm}
\noindent
Keywords: Generalized entropy, Polynomial entropy, Morse-Smale diffeomorphisms.
\\

\noindent
2020 Mathematics subject classification: 37B40, 37D15, 37E30  

 \section{Introduction}

A classical problem in dynamical systems is to measure the complexity of a map in terms of its orbits. One of the main tools we have to achieve this goal is topological entropy. The topological entropy of a map studies the exponential growth rate at which orbits are separated. Although it is a crucial tool for classifying highly chaotic dynamical systems, in many interesting families of dynamical systems, every system has vanishing entropy, and therefore, another tool is needed. J. P. Marco in \cite{Ma13} introduced the concept of polynomial entropy, and recently, the first author and E. Pujals introduced the notion of generalized topological entropy in \cite{CoPu21}. The later definition, instead of quantifying the complexity of a system with a single number, it works directly in the space of the orders of growth $\mathbb{O}$. 

J. P. Marco, C. Labrousse and P. Blanchard were the first to study dynamical systems with vanishing entropy from the perspective of polynomial entropy. First, J. P. Marco introduced the concept of polynomial entropy in \cite{Ma13}, and this topic was subsequently studied in \cite{La12}, \cite{La13-02}, \cite{LaMa14} and \cite{LaBe14} in the context of Hamiltonian dynamical systems and geodesic flows with zero-entropy. For maps on the interval, we have the works of C. Labrousse \cite{La13-01}, M. J. D. Carneiro and J. B. Gomes \cite{CaGo21} and S. Roth, Z. Roth and L. Snoha \cite{RoRoSn21} to refer to. Lastly, for general dynamical systems, we have the work of A. Artigue, D. Carrasco-Olivera and I. Monteverde \cite{ArCaMo17}. 

We would like to coment on two aspects of the work of L. Hauseux and F. Le Roux in \cite{HaRo19}. First, in \cite{CoPu21}, the authors observed that in the study of dynamical system with vanishing entropy, there are examples where the complexity of the overall system is greater than the complexity of the system restricted to its recurrent part. This phenomenon does not occur
in the presence of chaos due to the variational principle. It is our opinion that the definition of wandering entropy given in \cite{HaRo19} should give account of said jump. Second is with respect to the technique they developed to compute the polynomial entropy of a system whose non-wandering set consists of only one fixed point. In this article, we translate this technique to the context of generalized entropy and extend it to maps with a finite non-wandering set. We would like to mention that J. Kati\'{c} and M. Peri\'{c} in \cite{KaPe19} use said approach from \cite{HaRo19} to study Morse gradient systems with the singularities of a specific index. 

Let us consider $S$ a compact surface and $f:S\to S$ a Morse-Smale diffeomorphism. We recall that these maps are the ones whose non-wandering set is finite and, therefore, consists only of periodic points. These periodic points have to be hyperbolic; moreover, the intersection between any stable and any unstable manifolds is transverse. Let us suppose that $\Omega(f)=\cup_{i=1}^k \theta(p_i)$ where $\theta(p_i)$ denotes the orbit of point $p_i$. We can construct a finite graph $G=(V,E)$, where the vertices $V$ are the orbits $\theta(p_i)$ and the edges represent the existence of an orbit whose past is $\theta(p_i)$ and future is $\theta(p_j)$. Thus, we define
\[E=\{(\theta(p_i),\theta(p_j))\in V\times V: i\neq j \text{ and } W^u(\theta(p_i))\cap W^s(\theta(p_j))\neq \emptyset\}.\] 

If the graph $G$ had any cycle, then there would be a non-trivial homoclinic class, and therefore, the non-wandering set would not be finite. Now, let us define $L(f)$ as the length of the longest possible path in $G$. We represent the generalized entropy of $f$ by $o(f)$, and the polynomial entropy of $f$ by $h_{pol}(f)$.

\begin{maintheorem}\label{TeoMS}
Let $S$ be a compact surface and $f:S\to S$ a Morse-Smale diffeomorphism. Then, $o(f)=[n^{L(f)}]$. In particular, $h_{pol}(f) = L(f)$. 
\end{maintheorem}

As an immediate corollary, we have a rigidity result. 

\begin{corollary}\label{CorRig}
Let $S$ be a compact connected surface and $f:S\to S$ a Morse-Smale diffeomorphism such that $o(f)=[n]$. Then, $S=S^2$ and $f$ is a North-South dynamical system. 
\end{corollary}

We can apply theorem \ref{TeoMS} to understand the dispersion of orbits in the border of zero-entropy. Consider $Emb^r(\D^2)$ the space of 
$C^r$ embeddings of the disk $\D^2$ into itself. In \cite{GaStTr89}, J. M. Gambaudo, S. van Strein and C. Tresser constructed a system $f\in Emb^r(\D^2)$ that is Kupka-Smale and infinitely re-normalizable and happens to be in the border of chaos. This means that it can be perturbed into a system with positive classical entropy. The said example is also the limit of Morse-Smale diffeomorphisms and has a geometry that allow us to apply theorem \ref{TeoMS}. We would like to highlight that maps with said properties are far from being an isolated example. In \cite{CrPu16}, S. Crovisier and E. Pujals introduced the concept of mildly dissipative maps in $Emb^r(\D^2)$, and with C. Tresser in \cite{CrPuTr20}, they proved that any map with mild dissipation in the border of chaos is infinitely re-normalizable. In order for our following result to be a direct application of theorem \ref{TeoMS}, we need an additional hypothesis that we shall call standard geometry and explain in detail in subsection \ref{SubSecKS}. This property holds in the Gambaudo - van Strein - Tresser example as well as the Hénon-like maps with small Jacobian. The latter is proved by A. de Carvalho, M. Lyubich and M. Martens in \cite{CaLyMa05}. 

Let $\mathbb{P}$ stand for the family of polynomial orders of growth $[n^t]$ with $0<t<\infty$.  

\begin{corollary}\label{CorKS}
Let $f\in Emb^r(\D^2)$ be a Kupka-Smale and infinitely re-normalizable with standard geometry. Then, $o(f)= \sup(\Pol)$. In particular, $h_{pol}(f)=\infty$ and $h(f)=0$. Moreover, the Kupka-Smale property is dense among the mildly dissipative maps of the disk with zero-entropy and infinitely re-normalizable (with or without standard geometry).
\end{corollary}

Regarding the generalized entropy in the border of chaos, the standard geometry property is probably removable. However, a deeper understanding of the dispersion of wandering orbits near recurrent sets is needed. The Kupka-Smale property allows us to apply theorem \ref{TeoMS} directly, yet it may not be necessary. It seems intuitive for us that $o(f)= \sup(\Pol)$ should hold for every map in said set. 

We can also deduce that the map $f\mapsto o(f)$ is continuous in some sense. 

\begin{corollary}\label{CorCont}
Let us consider a sequence of maps $\{f_k\}_{k\in\N} \in Emb^r(\D^2)$ and $f\in Emb^r(\D^2)$, all of them mildly dissipative. Suppose that $f_k$ is Morse-Smale, while $f$ is Kupka-Smale and infinitely re-normalizable with standard geometry. If $\lim_k f_k = f$ in the $C^r$ topology, then $``\lim_k o(f_k)"= \sup\{o(f_k)\}=o(f)$. In particular, $\lim_k h_{pol}(f_k)=h_{pol}(f)$.
\end{corollary}

In retrospect, in the naturality of this result lies an indication of why theorem \ref{TeoMS} is also true. We know that the polynomial entropy of Morse-Smale maps is positive because it has wandering points. When we consider a map in the border of chaos, we naturally want that $h_{pol}(f)=\infty$ and if it is approximated by Morse-Smale, the polynomial entropy of them should grow toward infinity. It is expected that maps in the border of chaos show a reminiscent phenomenon of the period doubling cascade of bifurcations for interval maps. In this context, the only dynamical quantity that is growing and has a global understanding of the map is $L(f)$. 

Interestingly enough, in their recent work \cite{RoRoSn21}, S. Roth, Z. Roth and L. Snoha proved the counterpart of our results in the one-dimensional context. They showed that for endomorphisms of the interval with zero entropy, their polynomial entropy is equal to a quantity equivalent to our $L(f)$, and they obtained analogous results to our corollaries \ref{CorKS} and \ref{CorCont}. In particular, for the logistic map, they showed that in each bifurcation of the period doubling cascade, the polynomial entropy increases by $1$, as well as that the infinitely re-normalizable map has infinite polynomial entropy. 

The proof of theorem \ref{TeoMS} has two steps: understanding the geometrical configuration of a Morse-Smale and computing the generalized entropy through a nice codification of orbits. In the study of the separation of the orbits of wandering dynamics, Hauseux and Le Roux introduced the concept of wandering mutually singular sets in \cite{HaRo19}. We would like to point out that for dynamical systems with finite non-wandering set, this concept is in some sense the equivalent of the Markov partitions for uniformly hyperbolic systems.

We denote $\Omega(f)$ as the non-wandering set of $f$. We say that the subsets $Y_1,\cdots, Y_L$ of $M\setminus \{\Omega(f)\}$ are mutually singular if for every $N>0$, there exists a point $x$ and 
times $n_1,\cdots, n_L$ such that $f^{n_i}(x)\in Y_i$ for every $i=1,\cdots, L$, and $|n_i -n_j|>N$ for every $i\neq j$. 

\setcounter{maintheoremb}{1}
\begin{maintheoremb}\label{TeoCodB}
Let us consider $f:M\to M$ as a homeomorphism of a compact metric space whose non-wandering set is finite. Then, its generalized entropy is the supremum of the orders of growth associated with the codification of the orbits using mutually singular sets. 
\end{maintheoremb}

A more precise statement of this theorem is given in subsection \ref{SubCod}. We want to now discuss the interplay of theorems \ref{TeoMS}, \ref{TeoCodB} with the main theorem in \cite{HaRo19}. A consequence of theorem \ref{TeoCodB} is that the generalized entropy of a homeomorphism with a finite non-wandering set always verifies $[n]\leq o(f)\leq \sup(\mathbb{P})$. Let us consider $\Ha$ as a family of homeomorphism on a surface such that its non-wandering set is finite. Let us also consider $\Ha_1$ as the subset of $\Ha$ such that the surface is the sphere and its non-wandering set is only one fixed point. Part of the translation of the result in \cite{HaRo19} in our context is that for every $t\geq 2$, there exists $f\in \Ha_1$ such that $o(f) = [n^t]$. This result in conjunction with our previous observation can be summarized as follows:

\emph {For every $f\in \Ha$, $[n]\leq o(f)\leq \sup(\mathbb{P})$. Moreover, for every $t\geq 2$, there exists $f\in \Ha$ such that $o(f)=[n^t]$.} 

Theorem \ref{TeoMS} tells us that for those maps in $\Ha$ that are Morse-Smale diffeomorphisms, the set $\{o(f)\in \COG\}$ becomes discrete. It is not clear to us whether the differentiability plays a key role in this. We wonder if there is an obstruction for the examples built in \cite{HaRo19} to be differentiable and whether this contrast is another case of pathological differences between the continuous world and the differentiable one.  

A final comment in this topic is made toward the jump from $[n]$ to $[n^2]$ for the homeomorphisms in the sphere. In \cite{HaRo19}, the authors quote a result that seems to apply only for homeomorphisms in $\Ha_1$, and we would like to know if it possible to obtain a map in $\Ha$ for every $t\in (1,2)$ such that $o(f)=[n^t]$. 

We finish this introduction with a small discussion regarding the higher dimension. From theorem \ref{TeoCodB} and an intermediary lemma for theorem \ref{TeoMS}, we deduce the following result. 

\begin{corollary}\label{CorDimN}
Consider $M^n$ a compact manifold of dimension $n$ and $f:M\to M$ a Morse-Smale diffeomorphism. In that case, $o(f)\leq [n^{L(f)}]$. 
\end{corollary}

Regarding the other inequality, it is unclear to us if our technique holds. The main obstruction we see is the possible interaction between the periodic saddles of different indices. 

We would like to thank Enrique Pujals for the many discussions and insightful comments. 

This work is structured as follows:
\begin{itemize}
\item In section \ref{SecPreSta}, we provide a quick summary of the preliminaries on the topic of generalized entropy, mildly dissipative maps and the coding of orbits developed in \cite{HaRo19}. We also make a precise statement for theorem \ref{TeoCodB}.
\item In section \ref{SecTeoMS}, we prove theorem \ref{TeoMS}, assuming theorem \ref{TeoCodB}, and we prove corollaries \ref{CorRig}, \ref{CorKS}, \ref{CorCont} and \ref{CorDimN}.
\item In section \ref{SecTeoCod}, we prove theorem \ref{TeoCodB}.
\end{itemize}
 \section{Preliminaries and Statements}\label{SecPreSta}

\subsection{Orders of growth and generalized entropy}

Let us briefly recall how the complete set of the orders of growth and the generalized entropy of a map are defined in \cite{CoPu21}. First, we consider the space of non-decreasing sequences in $[0,\infty)$: $$\mathcal{O}=\{a:\mathbb{N}\rightarrow [0,\infty):a(n)\leq a(n+1),\, \forall n\in \mathbb{N}\}.$$
Next, we define the equivalence relationship $\approx$ in $\OR$ by $a(n)\approx b(n)$ if and only if there exist $c_1,c_2\in (0,\infty)$ such that $c_1 a(n)\leq b(n)\leq c_2 a(n)$ for all $n\in \mathbb{N}$. Since the two sequences are related, if both have the same order of growth, we call the quotient space $\displaystyle \mathbb{O}=\mathcal{O}/_{\approx}$ as the space of the orders of growth. If $a(n)$ belongs to $\mathcal{O}$, we are going to denote $[a(n)]$ as the associated class in $\mathbb{O}$. If a sequence is defined by a formula (for example, $n^2$), then the order of growth associated will be represented by the formula between the brackets ($[n^2]\in \mathbb{O}$).

Once $\mathbb{O}$ has been constructed, we define on it a very natural partial order. We say that $[a(n)] \leq [b(n)]$ if there exists $C>0$ such that $a(n) \leq Cb(n)$, for all $n\in\mathbb{N}$. With this, we consider $\overline{\mathbb{O}}$ the Dedekind-MacNeille completion. This is the smallest complete lattice that contains $\mathbb{O}$. In particular, it is uniquely defined and we will always consider that $\mathbb{O}\subset \overline{\mathbb{O}}$. We will also call $\overline{\mathbb{O}}$ the complete set of the orders of growth. 

Now, we proceed to define the generalized entropy of a dynamical system in the complete space of the orders of growth. Given $(M,d)$, a compact metric space and $f:M\rightarrow M$ a continuous map. We define in $M$ the distance 
\[d^{f}_{n}(x,y)=\max \{d(f^k(x),f^k(y)); 0\leq k \leq n-1\},\]
 and we denote the dynamical ball as $B(x,n,\e)=\{y\in M; d^{f}_{n}(x,y)<\e\}$. A set $G\subset M$ is a $(n,\e)$-generator if $\displaystyle M=\cup_{x\in G} B(x,n,\e)$. Given the compactness of M, there always exists a finite $(n,\e)$-generator set. Then, we define $g_{f,\e}(n)$ as the smallest possible cardinality of a finite $(n,\e)$-generator. If we fix $\e>0$, then $g_{f,\e}(n)$ is an increasing sequence of natural numbers,
and therefore, $g_{f,\e}(n) \in \mathcal{O}$. For a fixed $n$, if $\e_1<\e_2$, then $g_{f,\e_1} (n) \geq g_{f, \e_2}(n)$, and therefore, $[g_{f,\e_1}(n)]\geq [g_{f,\e_2}(n)]$ in $\mathbb{O}$. We consider the set $G_f=\{[g_{f,\e}(n)]\in \mathbb{O}:\e>0\}$, and the generalized entropy of $f$ as 
$$o(f)=\text{\textquotedblleft}\lim_{\e\rightarrow 0}"[g_{f,\e}(n)] =\sup G_f \in \overline{\mathbb{O}}. $$

This object is a dynamical invariant.
\begin{theorem}[Correa-Pujals]\label{TeoCoPu01}
	Let $M$ and $N$ be two compact metric spaces and $f:M\to M$, $g:N\to N$, two continuous maps. Suppose there exists $h:M\to N$, a homeomorphism, such that $h\circ f = g \circ h$. Then, $o(f)=o(g)$.
\end{theorem}

We also define the generalized entropy through the point of view of $(n,\e)$-separated. We say that $E\subset M$ is $(n,\e)$-separated if $B(x,n,\e)\cap E = \{x\}$, for all $x\in E$. We define $s_{f,\e}(n)$ as the maximal cardinality of a $(n,\e)$-separated set. Analogously, as with $g_{f,\e}(n)$, if we fix $\e>0$, then $s_{f,\e}(n)$ is a non-decreasing sequence of natural numbers. Again, for a fixed $n$, if $\e_1<\e_2$, then $s_{f,\e_1} (n) \geq s_{f, \e_2}(n)$, and therefore, $[s_{f,\e_1}(n)]\geq [s_{f,\e_2}(n)]$. If we consider the set $S_f=\{[s_{f,\e}(n)]\in \mathbb{O}:\e>0\}$, then 
$$o(f)=\sup S_f \in \overline{\mathbb{O}}. $$

For a final comment on this topic, the generalized entropy of a map can also be defined in compact subsets that may not be invariant. Given $K\subset M$, a compact subset, the definition of $g_{f,\e,K}(n)$ as the minimal number of $(n,\e)$-balls (centered at points in $K$) that are needed to cover $K$ also makes sense. With it, we can define $o(f,K)=\sup\{[g_{f,\e,K}(n)]:\e>0\}$ as the generalized entropy of $f$ in $K$.

Now, let us explain how the generalized topological entropy is related to the classical notion of topological entropy and polynomial entropy. Given a dynamical system $f$, recall that the topological entropy of $f$ is 
\[h(f) = \lim_{\e\to 0}\limsup_{n\rightarrow \infty} \frac{log(g_{f,\e}(n))}{n},\]
and the polynomial entropy of $f$ is
\[h_{pol}(f) = \lim_{\e\to 0}\limsup_{n\rightarrow \infty} \frac{log(g_{f,\e}(n))}{log(n)}.\]

We define the family of exponential orders of growth as the set $\mathbb{E}=\{[\exp(tn)]; t \in(0,\infty) \} \subset \mathbb{O}$ and the family of polynomials orders of growth as the set $\mathbb{P}=\{[n^t];t\in(0,\infty)\}$. In \cite{CoPu21}, the authors defined two natural projections, $\pi_{\mathbb{E}}:\overline{\mathbb{O}}\rightarrow [0,\infty]$ and $\pi_{\mathbb{P}}:\overline{\mathbb{O}}\rightarrow [0,\infty]$, for which the following theorem holds.

\begin{theorem}[Correa-Pujals]\label{teo122}
	Let $M$ be a compact metric space and $f : M \rightarrow M$, a continuous map. Then, $\pi_\Exp(o(f))=h(f)$ and $\pi_\Pol(o(f))=h_{pol}(f)$.
\end{theorem}

\subsection{Infinitely re-normalizable systems of the disk}\label{SubSecKS}
Let us begin with the definition of mildly dissipative systems. Throughout this subsection, we are going to work with maps in $Emb^r(\D^2)$, the space of $C^r$ embeddings of the disk $f:\D^2\to f(\D^2)\subset interior(\D^2)$. A map $f$ is dissipative if $|det(Df_x)|<1$ for every $x\in \D^2$. When $f$ is dissipative, every ergodic measure has at least one negative Lyapunov exponent. In this scenario, almost every point has a Pesin stable manifold $W^s(x)$ and we call $W^s_\D(x)$ the connected component in $\D^2$, which contains $x$. We say that $f$ is mildly dissipative if for every ergodic measure that is not supported on a hyperbolic sink, and for almost every point, $W^s_\D(x)$ splits the disk in two. 

Given $f\in Emb^r(\D^2)$, we say that it is infinitely re-normalizable if there exists a sequence of families of disks
$\mathcal{D}^m=\{D^m_1,\cdots, D^m_{l_m}\}$ such that
\begin{enumerate}
\item Every $D^{m+1}_i$ is contained in some $D^m_j$. 
\item For every $D^m_i$, there exists an integer $k^m_i$ such that $f^{k^m_i}(D^m_i)\subset D^m_i$.
\item For a fixed $m$, the discs $f^j(D^m_i)$ for $1\leq i \leq l_m$ and $0\leq j< k^m_i$ are pairwise disjoint. 
\item The times $k^m_i$ tends to infinity. 
\item There is no periodic point in the border of any disk $D^m_i$. 
\end{enumerate}

The standard geometry property is that the diameters of $D^m_i$ tend to $0$ uniformly.

For each $D^m_i$, let us consider $W^m_i= D^m_i\cup f(D^m_i)\cup \cdots f^{k^m_i-1}(D^m_i)$, and for each $m$, we define $V^m=W^m_1\cup\cdots \cup W^m_{l_m}$. In \cite{CrPuTr20}, the authors prove the following theorem.

\begin{theorem}[Crovisier-Pujals-Tresser]
 Given $f\in Emb^r(\D^2)$ mildly dissipative with zero topological entropy, if it is in the border of chaos, then it is infinitely re-normalizable. Moreover, the non-wandering set of $f$ outside $V^m$ consists of periodic points with a period smaller or equal to $\max\{k^m_i\}$.
\end{theorem}

\subsection{Coding of orbits}\label{SubCod}

Now, we proceed to extend the coding of the orbits done in \cite{HaRo19} to dynamical systems whose non-wandering set is finite, and we translate it to the language of generalized entropy simultaneously. Let us consider $M$ a compact metric space and $f:M\to M$ a homeomorphism such that $\Omega(f)=\{p_1,\cdots, p_k\}$. Let $\F$ be a finite family of non-empty subsets of $M\setminus \Omega(f)$. We denote by $\cup \F$ the union of all the elements of $\F$ and by $\infty_{\F}$ the complement of $\cup \F$. Let us fix a positive integer $n$ and consider $\underline{x}=(x_0,\cdots,x_{n-1})$ a finite sequence of points in $M$ and $\underline{w}=(w_0,\cdots,w_{n-1})$ a finite sequence of elements of $\F\cup \{\infty_{\F}\}$. We say that $\underline{w}$ is a coding of $\underline{x}$, relative to $\F$, if for every $i=0,\cdots, n-1$, we have $x_i\in w_i$. Whenever the family $\F$ is fixed, we simplify the notation by using $\infty$ instead of $\infty_{\F}$. Note that if the sets of $\F$ are not disjoint, we can have more than one coding for a given sequence. We denote the set of all the codings of all orbits $(x,f(x),\cdots, f^{n-1}(x))$ of length $n$ by $\A_{n}(f,\F)$. We define the sequence $c_{f,\F}(n)=\# \A_{n}(f,\F)$, and it is easy to see that $c_{f,\F}(n)\in \mathcal{O}$.

\begin{example}\label{ex1}
	 Let $T:\mathbb{R}^2\rightarrow \mathbb{R}^2$ be the translation $(x,y)\mapsto (x+1, y)$. To fit our setting, we consider the compactification of the plane by one point ($\infty$), obtaining $M=S^2$ and a map $f:S^2\rightarrow S^2$ such that $\Omega (f)=\{\infty\}$. As we have to work with the subsets of $M\setminus \{\infty\}$ for our coding, we may as well keep working on $\mathbb{R}^2$. Let $Y$ be a compact subset of $\mathbb{R}^2$ and let us suppose that its diameter is less than $1$. Then, we can see that $c_{f,\{Y\}}(n)=n$. Indeed, the elements of $\A_n(f,\{Y\})$ are exactly all the words of the form $(\infty,\cdots, \infty,Y,\infty,\cdots,\infty)$, and therefore, it contains $n$ elements. In case $Y$ has a diameter greater than one, it can easily be observed that the equality $[c_{f,\{Y\}}(n)]=[n]$ still holds.

\end{example}

\begin{example}[Reeb's flow/Brouwer's counter-example]\label{ex2}
Consider the map $H:\mathbb{R}^2\rightarrow \mathbb{R}^2$, the time-one map of Reeb's flow (see figure \ref{reeb}). We recall that this map is a classical example of a Brouwer homeomorphism that is not conjugated to a translation. Again, we can compactify $\mathbb{R}^2$ to the sphere $S^2$ and extend $H$ to the map $h:S^2\to S^2$ whose wandering set is only the fixed point at infinity. Again, we shall keep working in $\mathbb{R}^2$. Let $Y_1,Y_2$ be two disks, not containing the origin, whose interiors meet the lines $y=1$ and $y=-1$, respectively. To simplify the computation, we assume the disks are small enough for them to not intersect any of their respective forward images. Note that if $Y_1$ and $Y_2$ are chosen such that $f^n(Y_1)\cap Y_2\neq \emptyset$ for all $n$, then the elements of $\A_n(f,\{Y_1,Y_2\})$ are exactly all the words of the form $$(\infty,\cdots, \infty,Y_1,\infty,\cdots,\infty, Y_2,\infty,\cdots,\infty).$$  
Thus, $c_{f,\{Y_1,Y_2\}}(n)=n(n-1)/2$ and $[c_{h,\{Y_1,Y_2\}}(n)]=[n^2]$.		
\end{example}

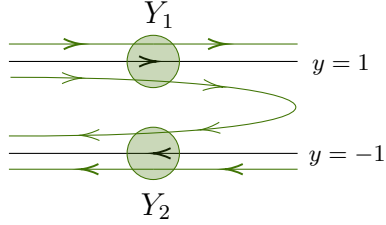
\begin{figure}[h!]
	\begin{center}
		\begin{tikzpicture}[x=0.75pt,y=0.75pt,yscale=-0.8,xscale=0.8]
			
			\draw    (30.5,61) -- (121.5,61) ;
			\draw    (123.5,61) -- (212.5,61) ;
			\draw [shift={(123.5,61)}, rotate = 180] [color={rgb, 255:red, 0; green, 0; blue, 0 }  ][line width=0.75]    (10.93,-3.29) .. controls (6.95,-1.4) and (3.31,-0.3) .. (0,0) .. controls (3.31,0.3) and (6.95,1.4) .. (10.93,3.29)   ;
			\draw    (30.5,119) -- (121.5,119) ;
			\draw    (123.5,119) -- (212.5,119) ;
			\draw [shift={(121.5,119)}, rotate = 0] [color={rgb, 255:red, 0; green, 0; blue, 0 }  ][line width=0.75]    (10.93,-3.29) .. controls (6.95,-1.4) and (3.31,-0.3) .. (0,0) .. controls (3.31,0.3) and (6.95,1.4) .. (10.93,3.29)   ;
			\draw [color={rgb, 255:red, 65; green, 117; blue, 5 }  ,draw opacity=1 ]   (30.5,50) -- (121.5,50) ;
			\draw [shift={(76,50)}, rotate = 180] [color={rgb, 255:red, 65; green, 117; blue, 5 }  ,draw opacity=1 ][line width=0.75]    (10.93,-3.29) .. controls (6.95,-1.4) and (3.31,-0.3) .. (0,0) .. controls (3.31,0.3) and (6.95,1.4) .. (10.93,3.29)   ;
			\draw [color={rgb, 255:red, 65; green, 117; blue, 5 }  ,draw opacity=1 ]   (121.5,50) -- (212.5,50) ;
			\draw [shift={(167,50)}, rotate = 180] [color={rgb, 255:red, 65; green, 117; blue, 5 }  ,draw opacity=1 ][line width=0.75]    (10.93,-3.29) .. controls (6.95,-1.4) and (3.31,-0.3) .. (0,0) .. controls (3.31,0.3) and (6.95,1.4) .. (10.93,3.29)   ;
			\draw [color={rgb, 255:red, 65; green, 117; blue, 5 }  ,draw opacity=1 ]   (30.5,129) -- (121.5,129) ;
			\draw [shift={(76,129)}, rotate = 0] [color={rgb, 255:red, 65; green, 117; blue, 5 }  ,draw opacity=1 ][line width=0.75]    (10.93,-3.29) .. controls (6.95,-1.4) and (3.31,-0.3) .. (0,0) .. controls (3.31,0.3) and (6.95,1.4) .. (10.93,3.29)   ;
			\draw [color={rgb, 255:red, 65; green, 117; blue, 5 }  ,draw opacity=1 ]   (121.5,129) -- (212.5,129) ;
			\draw [shift={(167,129)}, rotate = 0] [color={rgb, 255:red, 65; green, 117; blue, 5 }  ,draw opacity=1 ][line width=0.75]    (10.93,-3.29) .. controls (6.95,-1.4) and (3.31,-0.3) .. (0,0) .. controls (3.31,0.3) and (6.95,1.4) .. (10.93,3.29)   ;
			\draw [color={rgb, 255:red, 65; green, 117; blue, 5 }  ,draw opacity=1 ]   (31.5,71) .. controls (271.5,71) and (271.5,109) .. (30.5,108) ;
			\draw  [color={rgb, 255:red, 65; green, 117; blue, 5 }  ,draw opacity=1 ] (64,68) .. controls (68.17,69.94) and (72.33,71.11) .. (76.5,71.5) .. controls (72.33,71.89) and (68.17,73.06) .. (64,75) ;
			\draw  [color={rgb, 255:red, 65; green, 117; blue, 5 }  ,draw opacity=1 ] (87.56,110.89) .. controls (83.36,109.02) and (79.18,107.93) .. (75,107.61) .. controls (79.16,107.15) and (83.3,105.9) .. (87.44,103.89) ;
			\draw  [color={rgb, 255:red, 65; green, 117; blue, 5 }  ,draw opacity=1 ] (153.4,72.37) .. controls (157.34,74.73) and (161.36,76.33) .. (165.47,77.15) .. controls (161.28,77.11) and (157.02,77.83) .. (152.67,79.33) ;
			\draw  [color={rgb, 255:red, 65; green, 117; blue, 5 }  ,draw opacity=1 ] (163.83,106.33) .. controls (159.48,104.83) and (155.22,104.11) .. (151.03,104.15) .. controls (155.14,103.33) and (159.16,101.73) .. (163.1,99.37) ;
			\draw  [color={rgb, 255:red, 65; green, 117; blue, 5 }  ,draw opacity=1 ][fill={rgb, 255:red, 65; green, 117; blue, 5 }  ,fill opacity=0.28 ] (105,61) .. controls (105,51.89) and (112.39,44.5) .. (121.5,44.5) .. controls (130.61,44.5) and (138,51.89) .. (138,61) .. controls (138,70.11) and (130.61,77.5) .. (121.5,77.5) .. controls (112.39,77.5) and (105,70.11) .. (105,61) -- cycle ;
			\draw  [color={rgb, 255:red, 65; green, 117; blue, 5 }  ,draw opacity=1 ][fill={rgb, 255:red, 65; green, 117; blue, 5 }  ,fill opacity=0.28 ] (105,119) .. controls (105,109.89) and (112.39,102.5) .. (121.5,102.5) .. controls (130.61,102.5) and (138,109.89) .. (138,119) .. controls (138,128.11) and (130.61,135.5) .. (121.5,135.5) .. controls (112.39,135.5) and (105,128.11) .. (105,119) -- cycle ;
			
			\draw (220,55.4) node [anchor=north west][inner sep=0.75pt]  [font=\scriptsize]  {$y=1$};
			\draw (218,112.4) node [anchor=north west][inner sep=0.75pt]  [font=\scriptsize]  {$y=-1$};
			\draw (114,21.4) node [anchor=north west][inner sep=0.75pt]    {$Y_{1}$};
			\draw (112,142.4) node [anchor=north west][inner sep=0.75pt]    {$Y_{2}$};
		\end{tikzpicture}
	\end{center}
	\caption{Reeb's flow/Brouwer's counter-example.} \label{reeb}
\end{figure}

We say that a set $Y$ is wandering if $f^n(Y)\cap Y =\emptyset$ for every $n\geq 1$. We say that $Y$ is a compact neighborhood if it is compact and is the closure of an open set. We say that the subsets $Y_1,\cdots, Y_L$ of $M\setminus \{\Omega(f)\}$ are mutually singular if, for every $N>0$, there exists a point $x$ and times $n_1,\cdots, n_L$ such that $f^{n_i}(x)\in Y_i$ for every $i=1,\cdots, L$, and $|n_i -n_j|>N$ for every $i\neq j$. Note that in the previous example, sets $Y_1,Y_2$ are mutually singular.	

Let us call $\Sigma$ a family of finite families of wandering compact neighborhoods that are mutually singular. Given $\delta>0$, we define $\Sigma_\delta$ as the subset of $\Sigma$ formed by every family whose every element has a diameter smaller than $\delta$. Now, we present a more precise statement of theorem \ref{TeoCodB}.

\begin{maintheorem}\label{TeoCod}
Let $M$ be a compact metric space and $f:M\rightarrow M$ a homeomorphism such that $\Omega(f)$ is finite. Then, 
\[o(f)=\sup\{[c_{f,\F}(n)]\in \OG: \F\in \Sigma\}.\]
In addition, the equation also holds if we switch $\Sigma$ by $\Sigma_\delta$. 
\end{maintheorem}

 \section{Proof of theorem \ref{TeoMS}}\label{SecTeoMS}

We use theorem \ref{TeoCod} to prove theorem \ref{TeoMS}. However, the proof of theorem \ref{TeoCod} is more technical and since the proof of theorem \ref{TeoMS} has more geometrical components, we choose to prove \ref{TeoMS} first and show the proof of theorem \ref{TeoCod} later. At the end of this section, we provide a proof of corollaries \ref{CorRig}, \ref{CorKS}, \ref{CorCont} and \ref{CorDimN}.

Let us consider $S$ a compact surface  and $f:S\to S$ a Morse-Smale diffeomorphism. As we are in dimension two, we have three possibilities for the hyperbolic periodic points. They are either sinks, sources or saddles of index 1. Using theorem \ref{TeoCod}, we need to first compute $[c_{f,\F}(n)]$ for finite families $\F$ of disjoint, mutually singular compact neighborhoods, and then take the supremum over such families. Our following lemma tells the location of the sets of any $\F\in \Sigma$. 

\begin{lemma}\label{LemLoc}
Given $\F = \{Y_1,\cdots, Y_L\}\in \Sigma$, there exist $\theta(p_1),\cdots, \theta(p_{L})$ orbits of periodic points such that $Y_i\cap W^s(\theta(p_i))\neq \emptyset$ and $Y_{i+1}\cap W^u(\theta(p_i))\neq \emptyset$. Moreover, if the diameter of the elements of $\F$ is small enough, then $W^u(\theta(p_i))\cap W^s(\theta(p_{i+1}))\neq \emptyset$. 
\end{lemma}

The following figure represents the above-mentioned lemma. 

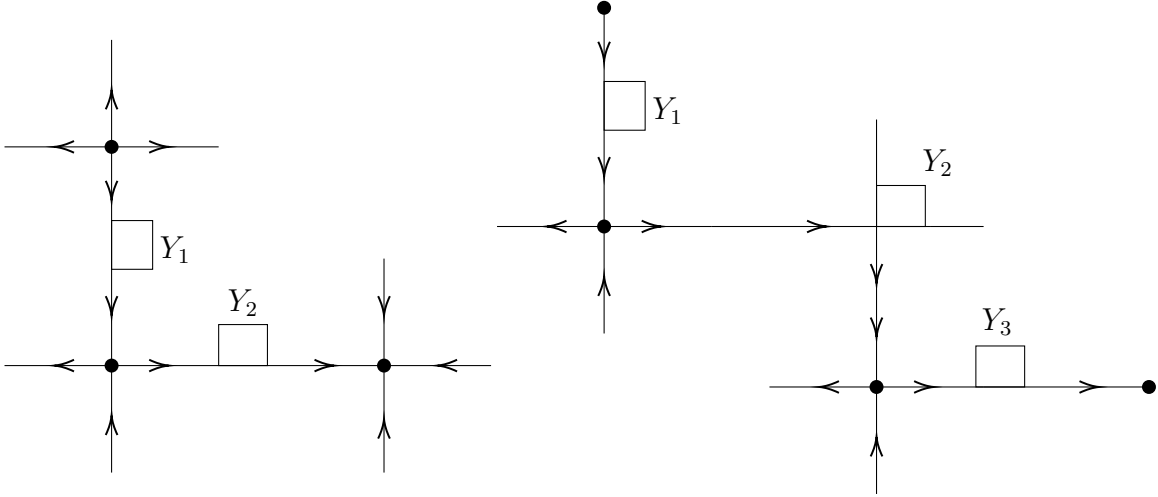
\input{fig04}

\begin{proof}
The property of $\F$ to be mutually singular creates long segments of orbits whose endpoints belong to $Y_i$ and $Y_{i+1}$. The accumulation of such segments contains an invariant set; therefore, it must contain a periodic orbit $\theta(p_i)$. Now, the endpoints in  $Y_i$ must accumulate over the stable manifold of $\theta(p_i)$, and the endpoint in $Y_{i+1}$ must accumulate over the unstable manifold of $\theta(p_i)$. As all the elements of $\F$ are compact, we conclude the first part of the lemma. 

Now, we shall prove the second part of the lemma. Note that if the unstable manifold of a periodic saddle $\theta(p)$ does not intersect the stable manifold of $\theta(q)$, yet it accumulates it, then any point in said stable manifold would be non-wandering. This can be seen with the help of Hartman-Grobman's theorem. As $f$ is Morse-Smale, this can not happen. We thus conclude the existence of a positive $\delta$ such that if $dist(W^u(\theta(p)), W^s(\theta(q))) \leq \delta$, then $W^u(\theta(p))\cap W^s(\theta(q))\neq \emptyset$. As we can choose the elements of $\F$ with an arbitrarily small diameter, we conclude the proof of the lemma.   
\end{proof}

Given $\F\in \Sigma$, take $L=\#\F$ and consider the periodic orbits $\theta(p_1),\cdots, \theta(p_L)$ from lemma \ref{LemLoc}. Note that any point in the intersection between $Y_1$ and the stable manifold of $\theta(p_1)$ must also belong to the unstable manifold of some periodic orbit $\theta(p_0)$. Therefore, $\theta(p_0),\theta(p_1),\cdots, \theta(p_L)$ induces a path of length $L$ in the graph constructed in the introduction, particularly $L\leq L(f)$.

\begin{lemma}\label{LemUppBou} 
Let us consider $M$ a compact metric space and $f:M\to M$ a homeomorphism. Given $\F\in \Sigma$, take $L=\#\F$. Then, $[c_{f,\F}(n)]\leq [n^{L}]$. 
\end{lemma}

By applying the above lemma to our previous reasoning, we infer that $[c_{f,\F}(n)]\leq [n^{L}]\leq [n^{L(f)}]$, and by theorem \ref{TeoCod}, we conclude $o(f)\leq [n^{L(f)}]$. Thus, it only remains for us to prove the other inequality, for which we only need to find a family $\F\in \Sigma$ with $L(f)$ elements such that $[c_{f,\F}(n)]= [n^{L(f)}]$. Before moving on, we prove lemma \ref{LemUppBou}.

\begin{proof}[Proof lemma \ref{LemUppBou}]
For a given subset $J=\{i_1,\cdots, i_l\}\subset \{1,\cdots,L\}$, we define $V^n_J$ as the set of all the words of the form 
\[(\infty,\cdots, \infty,Y_{i_1},\infty,\cdots,\infty, Y_{i_2},\infty,\cdots,\infty, Y_{i_l},\infty,\cdots, \infty),\]
with $n$ letters. By definition of $\A_n(f,\F)$, we know that 
\[\A_n(f,\F)\subset \bigcup_{J} V^n_J.\]
Each set of $V^n_J$ represents the combinations of $n$ distinct objects taken $\#J$ at a time. Therefore, $\#V^n_J =\frac{n!}{(n-\#J)!(\#J)!}$, which is a polynomial in $n$ of degree $\#J$. As there are finite $V^n_J$ (for a fixed $n$) and the largest possible degree is $L$, we conclude that $[c_{f,\F}(n)]\leq [n^{L}]$. 
\end{proof}

In the context of a homeomorphism such that its non-wandering set is finite, it is easier to work with the assumption that every periodic point is indeed a fixed point. However, to make such an assumption, we need to develop two notions: the bounded jump property and $N$-combinatorially complete. 

We say that a class of orders of growth $[a(n)]$ verifies the bounded jump property if there exists a constant $C>0$ such that $a(n+1)\leq Ca(n)$. Note that this definition does not depend on the choice of the class representative. We recall that a set $B\subset \N$ is syndetic if there exists $N\in \N$ such that for all $n$, the interval $[n,n+N]$ contains at least one point of $B$. 

\begin{lemma}\label{LemBouJum}
Let us consider $M$ a compact metric space and $f:M\to M$ a homeomorphism whose non-wandering set is finite. Given $\F\in \Sigma$, the class $[c_{f,\F}(n)]$ verifies the bounded jump property.
\end{lemma}

\begin{lemma}\label{LemSyn}
Let us consider $[a(n)]\in \OG$ an order of growth that verifies the bounded jump property, a syndetic set $B$ and a sequence $b(n)\in \OR$. If there exist two constants $c_1$ and $c_2$ such that $c_1 b(n)\leq a(n)\leq c_2 b(n)$ for all $n\in B$, then $[a(n)] =[b(n)]$. 
\end{lemma}

The first lemma tell us that $[c_{f,\F}(n)]$ always verifies the bounded jump property. The second lemma tell us that if we know an order of growth in a syndetic set and said order of growth verifies the bounded jump property, then we understand the order of growth in $\N$. For example, if we prove that $c_1 n^L\leq c_{f,\F}(n)\leq c_2 n^L$ for every  $n\in B$, then $[c_{f,\F}(n)]=[n^L]$. 

\begin{proof}[Proof of lemma \ref{LemBouJum}]
We consider $\F\in \Sigma$ and $L=\#\F$. Given $n\in \N$, we define the map $\varphi_n:\A_{n+1}(f,\F)\to \A_{n}(f,\F)$ that removes the last letter from the word $w$. The map $\varphi_n$ is surjective, and every $w\in \A_{n}(f,\F)$ has at most $L+1$ pre-images. Therefore,
\[c_{f,\F}(n+1)\leq (L+1)c_{f,\F}(n).\]
\end{proof}

\begin{proof}[Proof of lemma \ref{LemSyn}]
Let us fix $n\in \N$ and choose $n_1,n_2\in B$ such $n-N \leq n_1 \leq n \leq n_2 \leq n+N$. Observe that 
\[\frac{a(n)}{b(n)}\leq C^N \frac{a({n_1})}{b(n_1)}\leq c_2 C^N,\]
and
\[\frac{a(n)}{b(n)}\geq \frac{1}{C^N}\frac{a({n_2})}{b(n_2)}\geq \frac{c_1}{C^N}.\]
From both equations, we conclude $[a(n)] =[b(n)]$.
\end{proof}

Given $\F\in \Sigma$, we say that it is $N$-combinatorially complete if there exist $k_1,\cdots, k_{L-1}$ such that for every $n_1\geq k_1,\cdots, n_{L-1}\geq k_{L-1}$ with $n_i-k_i$ divisible by $N$,  there are points $x_1\in Y_1,\cdots, x_L\in Y_L$ with the property $x_{i+1} = f^{n_i}(x_i)$. If $\F$ is 1-combinatorially complete, we say that it is combinatorially complete.  

Let us consider $N$ such that every periodic point of $f^N$ is a fixed one. Note that if $\F\in \Sigma(f^N)$, then $\F$ also belongs to $\Sigma(f)$. Our following two lemmas explain why we can work with a map $f$ that has only fixed points.  

\begin{lemma}\label{LemNto1}
Given $N\in \N$, if $\F$ is combinatorially complete for $f^N$, then $\F$ is $N$-combinatorially complete for $f$.
\end{lemma}

\begin{lemma}\label{LemCom}
Given $\F\in\Sigma$, we take $L=\#\F$. If $\F$ is $N$-combinatorially complete for $f$, then $[c_{f,\F}(n)] = [n^L]$.
\end{lemma}

We proceed now to prove both lemmas. 

\begin{proof}[Proof of lemma \ref{LemNto1}]
Suppose that $\F=\{Y_1,\cdots, Y_L\}\in \Sigma$ is combinatorially complete for $f^N$. Take $\hat{k_1},\cdots, \hat k_{L-1}$ associated to said property and define $k_i = \hat k_i N$. For each $i$, we take $n_i$ such that $n_i - k_i$ is divisible by $N$. This is $n_i-k_i = Na_i$ for some $a_i$. As $\F$ is combinatorially complete for $f^N$, there exist $x_1\in Y_1,\cdots, x_L \in Y_L$ such that $(f^N)^{a_i+\hat k_i}(x_i) = x_{i+1}$ for every $i=1,\cdots, L-1$. It is simple to see that $(f^N)^{a_i+\hat k_i} = f^{n_i}$, and therefore, $\F$ is $N$-combinatorially complete for $f$. 
\end{proof}

\begin{proof}[Proof of lemma \ref{LemCom}]
Suppose that $\F$ is $N$-combinatorially complete and take $K= k_1+\cdots + k_{L-1}$. Consider the syndetic set $B=N\N + K+L$. For each $n\in B$, we define $\mathcal{B}_n(f,\F)$ as the set of words in $\A_n(f,\F)$, which verifies the following:
\begin{itemize}
\item Each word $w\in\mathcal{B}_n(f,\F)$ is associated with a segment of an orbit of a point from the $N$-combinatorially complete property.
\item Every letter $Y_1,\cdots, Y_L$ appears once. 
\item If $a$ is such that $n=aN + K + L$, then for every $w\in\mathcal{B}_n(f,\F)$, there exists $a_0,\cdots a_L$ such that $a=a_0+\cdots + a_L$ and $w$ has the form
\begin{center}   

\begin{tikzpicture}[x=0.75pt,y=0.75pt,yscale=-1,xscale=1]

	\draw (25,110.4) node [anchor=north west][inner sep=0.75pt]  [font=\scriptsize]  {$w=\left(\underbrace{\infty ,\cdots ,\infty }_{Na_{0}} ,Y_{1} ,\underbrace{\infty ,\cdots ,\infty }_{k_{1}} ,\underbrace{\infty ,\cdots ,\infty }_{Na_{1}} ,Y_{2} ,\infty ,\cdots ,\infty ,Y_{L-1} ,\underbrace{\infty ,\cdots ,\infty }_{k_{L-1}} ,\underbrace{\infty ,\cdots ,\infty }_{Na_{L-1}} ,Y_{L} ,\ \underbrace{\infty ,\cdots ,\infty }_{Na_{L}}\right).$};

\end{tikzpicture}
\end{center}
\end{itemize}

The property of $N$-combinatorially complete in conjunction with the proper selection of a segment of an orbit tells us that for any $a_0,\cdots, a_L$ with $a= a_0+\cdots + a_L$, there exists $w\in \mathcal{B}_n(f,\F)$ with the form mentioned in the previous equation. 
Let us consider $V_a =  V^{a+L}_{\{1,\cdots, L\}}$ as in the proof of lemma \ref{LemUppBou} and the map $\varphi:\mathcal{B}_n(f,\F)\to V_a$ defined as follows:
\begin{center}

	 \tikzset{every picture/.style={line width=0.75pt}} 
	 
	 \begin{tikzpicture}[x=0.75pt,y=0.75pt,yscale=-1,xscale=1]
	 	
	 	\draw    (29,80) -- (29,115) ;
	 	\draw [shift={(29,117)}, rotate = 270] [color={rgb, 255:red, 0; green, 0; blue, 0 }  ][line width=0.75]    (10.93,-3.29) .. controls (6.95,-1.4) and (3.31,-0.3) .. (0,0) .. controls (3.31,0.3) and (6.95,1.4) .. (10.93,3.29)   ;
	 	
	 	\draw (22,33.4) node [anchor=north west][inner sep=0.75pt]  [font=\footnotesize]  {$w=\left(\underbrace{\infty ,\cdots ,\infty }_{Na_{0}} ,Y_{1} ,\underbrace{\infty ,\cdots ,\infty }_{k_{1}} ,\underbrace{\infty ,\cdots ,\infty }_{Na_{1}} ,Y_{2} ,\infty ,\cdots ,\infty ,Y_{L-1} ,\underbrace{\infty ,\cdots ,\infty }_{k_{L-1}} ,\underbrace{\infty ,\cdots ,\infty }_{Na_{L-1}} ,Y_{L} ,\ \underbrace{\infty ,\cdots ,\infty }_{Na_{L}}\right)$};
	 	\draw (21,106.4) node [anchor=north west][inner sep=0.75pt]  [font=\footnotesize]  {$\varphi ( w) =\left(\underbrace{\infty ,\cdots ,\infty }_{a_{0}} ,Y_{1} ,\underbrace{\infty ,\cdots ,\infty }_{a_{1}} ,Y_{2} ,\infty ,\cdots ,\infty ,Y_{L-1} ,\underbrace{\infty ,\cdots ,\infty }_{a_{L-1}} ,Y_{L} ,\ \underbrace{\infty ,\cdots ,\infty }_{a_{L}}\right)$};

	 \end{tikzpicture}
	 
\end{center}

Our previous observation implies that $\varphi$ is a bijection. We know that $\# V_a = \frac{(a+L)!}{a!L!}$, which is a polynomial in $a$ of degree $L$, and as $a= \frac{n-K-L}{N}$ where $K,L$ and $N$ are constants, we conclude that $\#\mathcal{B}_n(f,\F)$ is a polynomial in $n$ of degree $L$. From this, we deduce the existence of a constant $c_1$ such that $c_1 n^L \leq \#\mathcal{B}_n(f,\F) \leq c_{f,\F}(n)$ for all $n\in B$. On the other hand, by lemma \ref{LemUppBou}, we know there exists $c_2$ such that $c_{f,\F}(n)\leq c_2 n^L$ for all $n\in \N$; finally, by lemma \ref{LemSyn}, we conclude that $[n^L]= [c_{f,\F}(n)]$.
\end{proof}

We are now in a condition to prove theorem \ref{TeoMS}

\begin{proof}[Proof of theorem \ref{TeoMS}] 
As explained before, as a consequence of lemma \ref{LemLoc}  and lemma \ref{LemUppBou}, we know that $o(f)\leq [n^{L(f)}]$. We take $N$ such that every periodic point of $f^N$ is indeed a fixed point. By lemmas \ref{LemNto1} and \ref{LemCom}, if we construct a family $\F\in \Sigma$ that is combinatorially complete for $f^N$ and has $L(f)$ elements, then we have finished. Note that lemma  \ref{LemLoc} tell us from where we should pick the elements of $\F$. 

To simplify our notation, we may assume that $N=1$. That is, from now on, map $f$ contains only hyperbolic fixed points in the non-wandering set. Let us take $p_0,\cdots,p_{L(f)}\in \Omega(f)$ such that $W^u(p_i)\cap W^s(p_{i+1})\neq \emptyset$ for all $i=0,\cdots, L(f)-1$. From this, it is simple to observe that $p_0$ is a source, $p_{L(f)}$ is a sink and every other $p_i$ is a saddle. Otherwise, we could extend the path, which is absurd based on our choice of $L(f)$. 

Let us begin by considering $U_1,\cdots, U_{L(f)-1}\subset S$, which are the linearizing neighborhoods of $p_1,\cdots,\ p_{L(f)-1}$, respectively. Note that we can assume them to be homeomorphic to a ball in $\R^2$. Given $p_i$ a saddle point, there is one connected component of $W^s(p_i)\setminus \{p_i\}$ that intersects $W^u(p_{i-1})$ and one connected component of  $W^u(p_i)\setminus \{p_i\}$ that intersects $W^s(p_{i+1})$. Both of these components define a quadrant in $U_i$, which we shall call $Q_i$. 

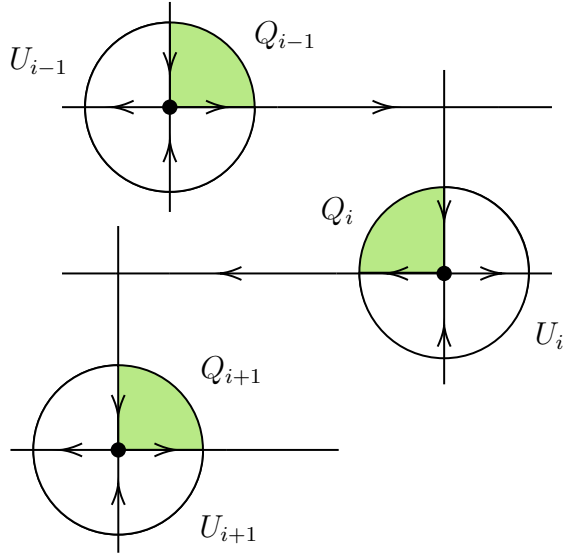
\input{fig07}

Now, for each $i=1,\cdots, L(f)-1$, we take $y_i$ a point in the intersection between $W^u(p_{i-1})\cap W^s(p_i)$ that also belongs to $Q_i$. For said $y_i$, we consider $H^1_i$ and $V^1_i$ as two small curves such that:
\begin{itemize}
\item both of them have $y_i$ as an endpoint.
\item $H^1_i\subset W^u(p_{i-1})\cap Q_i$ 
and $V^1_i \subset W^s(p_i)$ and for some $m_i>0$ $f^{-m_i}(V^1_i)\subset Q_{i-1}$.
\end{itemize}
  
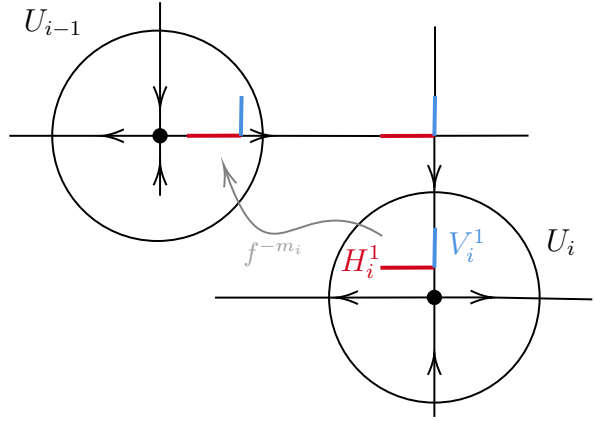
\begin{center}

\tikzset{every picture/.style={line width=0.75pt}} 

\begin{tikzpicture}[x=0.75pt,y=0.75pt,yscale=-1,xscale=1]
	
	\draw    (200.49,47) -- (200.49,98.43) ;
	\draw    (200.49,100.43) -- (200.49,114.43) ;
	\draw [shift={(200.49,100.43)}, rotate = 270] [color={rgb, 255:red, 0; green, 0; blue, 0 }  ][line width=0.75]    (10.93,-3.29) .. controls (6.95,-1.4) and (3.31,-0.3) .. (0,0) .. controls (3.31,0.3) and (6.95,1.4) .. (10.93,3.29)   ;
	\draw    (200.49,114.43) -- (173.3,114.43) ;
	\draw    (171.3,114.43) -- (124.5,114.43) ;
	\draw [shift={(171.3,114.43)}, rotate = 360] [color={rgb, 255:red, 0; green, 0; blue, 0 }  ][line width=0.75]    (10.93,-3.29) .. controls (6.95,-1.4) and (3.31,-0.3) .. (0,0) .. controls (3.31,0.3) and (6.95,1.4) .. (10.93,3.29)   ;
	\draw    (200.49,114.43) -- (227.67,114.43) ;
	\draw    (227.67,114.43) -- (254.86,114.43) ;
	\draw    (200.49,144.67) -- (200.49,125.89) ;
	\draw    (338.84,165.85) -- (338.84,114.43) ;
	\draw [shift={(338.84,140.14)}, rotate = 270] [color={rgb, 255:red, 0; green, 0; blue, 0 }  ][line width=0.75]    (10.93,-3.29) .. controls (6.95,-1.4) and (3.31,-0.3) .. (0,0) .. controls (3.31,0.3) and (6.95,1.4) .. (10.93,3.29)   ;
	\draw    (311.65,114.53) -- (256.86,114.44) ;
	\draw [shift={(256.86,114.44)}, rotate = 180.1] [color={rgb, 255:red, 0; green, 0; blue, 0 }  ][line width=0.75]    (10.93,-3.29) .. controls (6.95,-1.4) and (3.31,-0.3) .. (0,0) .. controls (3.31,0.3) and (6.95,1.4) .. (10.93,3.29)   ;
	\draw    (338.84,165.85) -- (338.84,196.01) ;
	\draw    (289.34,196.1) -- (228.19,196.1) ;
	\draw [shift={(289.34,196.1)}, rotate = 360] [color={rgb, 255:red, 0; green, 0; blue, 0 }  ][line width=0.75]    (10.93,-3.29) .. controls (6.95,-1.4) and (3.31,-0.3) .. (0,0) .. controls (3.31,0.3) and (6.95,1.4) .. (10.93,3.29)   ;
	\draw    (338.84,196.01) -- (366.02,196.01) ;
	\draw    (368.02,196.05) -- (418.5,197) ;
	\draw [shift={(368.02,196.05)}, rotate = 181.08] [color={rgb, 255:red, 0; green, 0; blue, 0 }  ][line width=0.75]    (10.93,-3.29) .. controls (6.95,-1.4) and (3.31,-0.3) .. (0,0) .. controls (3.31,0.3) and (6.95,1.4) .. (10.93,3.29)   ;
	\draw    (338.84,256.33) -- (338.84,226.17) ;
	\draw    (200.49,114.43) -- (200.49,128.44) ;
	\draw [shift={(200.49,128.44)}, rotate = 90] [color={rgb, 255:red, 0; green, 0; blue, 0 }  ][line width=0.75]    (10.93,-3.29) .. controls (6.95,-1.4) and (3.31,-0.3) .. (0,0) .. controls (3.31,0.3) and (6.95,1.4) .. (10.93,3.29)   ;
	\draw [shift={(200.49,114.43)}, rotate = 90] [color={rgb, 255:red, 0; green, 0; blue, 0 }  ][fill={rgb, 255:red, 0; green, 0; blue, 0 }  ][line width=0.75]      (0, 0) circle [x radius= 3.35, y radius= 3.35]   ;
	\draw    (338.84,196.01) -- (338.84,224.17) ;
	\draw [shift={(338.84,224.17)}, rotate = 90] [color={rgb, 255:red, 0; green, 0; blue, 0 }  ][line width=0.75]    (10.93,-3.29) .. controls (6.95,-1.4) and (3.31,-0.3) .. (0,0) .. controls (3.31,0.3) and (6.95,1.4) .. (10.93,3.29)   ;
	\draw [shift={(338.84,196.01)}, rotate = 90] [color={rgb, 255:red, 0; green, 0; blue, 0 }  ][fill={rgb, 255:red, 0; green, 0; blue, 0 }  ][line width=0.75]      (0, 0) circle [x radius= 3.35, y radius= 3.35]   ;
	\draw [color={rgb, 255:red, 0; green, 0; blue, 0 }  ,draw opacity=1 ][line width=0.75]    (339.18,94.33) -- (339.16,59.22) ;
	\draw [color={rgb, 255:red, 0; green, 0; blue, 0 }  ,draw opacity=1 ][line width=0.75]    (386.33,114.34) -- (338.84,114.43) ;
	\draw [color={rgb, 255:red, 0; green, 0; blue, 0 }  ,draw opacity=1 ][line width=0.75]    (338.84,196.01) -- (291.34,196.1) ;
	\draw [color={rgb, 255:red, 208; green, 2; blue, 27 }  ,draw opacity=1 ][line width=1.5]    (241.27,114.34) -- (214.08,114.43) ;
	\draw [color={rgb, 255:red, 74; green, 144; blue, 226 }  ,draw opacity=1 ][line width=1.5]    (241.27,114.43) -- (241.6,94.33) ;
	\draw   (146.11,114.43) .. controls (146.11,85.11) and (169.88,61.34) .. (199.2,61.34) .. controls (228.53,61.34) and (252.3,85.11) .. (252.3,114.43) .. controls (252.3,143.76) and (228.53,167.53) .. (199.2,167.53) .. controls (169.88,167.53) and (146.11,143.76) .. (146.11,114.43) -- cycle ;
	\draw   (285.74,196.01) .. controls (285.74,166.69) and (309.51,142.92) .. (338.84,142.92) .. controls (368.16,142.92) and (391.93,166.69) .. (391.93,196.01) .. controls (391.93,225.33) and (368.16,249.1) .. (338.84,249.1) .. controls (309.51,249.1) and (285.74,225.33) .. (285.74,196.01) -- cycle ;
	\draw [color={rgb, 255:red, 208; green, 2; blue, 27 }  ,draw opacity=1 ][line width=1.5]    (338.84,114.43) -- (311.65,114.53) ;
	\draw [color={rgb, 255:red, 74; green, 144; blue, 226 }  ,draw opacity=1 ][line width=1.5]    (338.84,114.43) -- (339.18,94.33) ;
	\draw [color={rgb, 255:red, 208; green, 2; blue, 27 }  ,draw opacity=1 ][line width=1.5]    (338.84,180.93) -- (311.65,181.02) ;
	\draw [color={rgb, 255:red, 74; green, 144; blue, 226 }  ,draw opacity=1 ][line width=1.5]    (338.84,180.93) -- (339.18,160.83) ;
	\draw [color={rgb, 255:red, 128; green, 128; blue, 128 }  ,draw opacity=1 ]   (233.19,131.06) .. controls (255.78,197.77) and (269.93,137.28) .. (312.5,165) ;
	\draw [shift={(232.5,129)}, rotate = 71.81] [color={rgb, 255:red, 128; green, 128; blue, 128 }  ,draw opacity=1 ][line width=0.75]    (10.93,-3.29) .. controls (6.95,-1.4) and (3.31,-0.3) .. (0,0) .. controls (3.31,0.3) and (6.95,1.4) .. (10.93,3.29)   ;
	
	\draw (290,168.24) node [anchor=north west][inner sep=0.75pt]  [color={rgb, 255:red, 208; green, 2; blue, 27 }  ,opacity=1 ]  {$H_{i}^{1}$};
	\draw (344.33,160.08) node [anchor=north west][inner sep=0.75pt]  [color={rgb, 255:red, 74; green, 144; blue, 226 }  ,opacity=1 ]  {$V_{i}^{1}$};
	\draw (130.25,48.9) node [anchor=north west][inner sep=0.75pt]    {$U_{i-1}$};
	\draw (393.75,159.9) node [anchor=north west][inner sep=0.75pt]    {$U_{i}$};
	\draw (241,165.4) node [anchor=north west][inner sep=0.75pt]  [font=\footnotesize,color={rgb, 255:red, 155; green, 155; blue, 155 }  ,opacity=1 ]  {$f^{-m_{i}}$};

\end{tikzpicture}

\end{center}

Once $H^1_i$ and $V^1_i$ are defined, we construct $H^2_i$ and $V^2_i$ as two curves and $Y_i$, thus verifying the following:
\begin{itemize}
\item The border of $Y_i$ is the union of $H^1_i$, $V^1_i$, $H^2_i$ and $V^2_i$.
\item $Y_i\subset Q_i$. 
\item The angle between the curves in each endpoint is not 0, $H^1_i$ does not meet $H^2_i$, and $V^1_i$ does not meet $V^2_i$. Therefore, $Y_i$ is a "rectangle".
\item There exist some $m_i$ such that $Z_i = f^{-m_i}(Y_i)\subset Q_{i-1}$.   
\end{itemize}

For the last point to hold, we shrink the curves (and, therefore, $Y_i$) if necessary. 

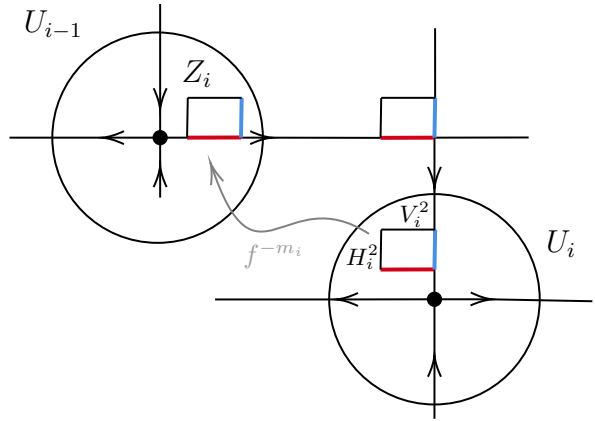
\begin{center}

\tikzset{every picture/.style={line width=0.75pt}} 

\begin{tikzpicture}[x=0.75pt,y=0.75pt,yscale=-1,xscale=1]
	
	\draw    (210.49,33) -- (210.49,84.43) ;
	\draw    (210.49,86.43) -- (210.49,100.43) ;
	\draw [shift={(210.49,86.43)}, rotate = 270] [color={rgb, 255:red, 0; green, 0; blue, 0 }  ][line width=0.75]    (10.93,-3.29) .. controls (6.95,-1.4) and (3.31,-0.3) .. (0,0) .. controls (3.31,0.3) and (6.95,1.4) .. (10.93,3.29)   ;
	\draw    (210.49,100.43) -- (183.3,100.43) ;
	\draw    (181.3,100.43) -- (134.5,100.43) ;
	\draw [shift={(181.3,100.43)}, rotate = 360] [color={rgb, 255:red, 0; green, 0; blue, 0 }  ][line width=0.75]    (10.93,-3.29) .. controls (6.95,-1.4) and (3.31,-0.3) .. (0,0) .. controls (3.31,0.3) and (6.95,1.4) .. (10.93,3.29)   ;
	\draw    (210.49,100.43) -- (237.67,100.43) ;
	\draw    (237.67,100.43) -- (264.86,100.43) ;
	\draw    (210.49,130.67) -- (210.49,111.89) ;
	\draw    (348.84,151.85) -- (348.84,100.43) ;
	\draw [shift={(348.84,126.14)}, rotate = 270] [color={rgb, 255:red, 0; green, 0; blue, 0 }  ][line width=0.75]    (10.93,-3.29) .. controls (6.95,-1.4) and (3.31,-0.3) .. (0,0) .. controls (3.31,0.3) and (6.95,1.4) .. (10.93,3.29)   ;
	\draw    (321.65,100.53) -- (266.86,100.44) ;
	\draw [shift={(266.86,100.44)}, rotate = 180.1] [color={rgb, 255:red, 0; green, 0; blue, 0 }  ][line width=0.75]    (10.93,-3.29) .. controls (6.95,-1.4) and (3.31,-0.3) .. (0,0) .. controls (3.31,0.3) and (6.95,1.4) .. (10.93,3.29)   ;
	\draw    (348.84,151.85) -- (348.84,182.01) ;
	\draw    (299.34,182.1) -- (238.19,182.1) ;
	\draw [shift={(299.34,182.1)}, rotate = 360] [color={rgb, 255:red, 0; green, 0; blue, 0 }  ][line width=0.75]    (10.93,-3.29) .. controls (6.95,-1.4) and (3.31,-0.3) .. (0,0) .. controls (3.31,0.3) and (6.95,1.4) .. (10.93,3.29)   ;
	\draw    (348.84,182.01) -- (376.02,182.01) ;
	\draw    (378.02,182.05) -- (428.5,183) ;
	\draw [shift={(378.02,182.05)}, rotate = 181.08] [color={rgb, 255:red, 0; green, 0; blue, 0 }  ][line width=0.75]    (10.93,-3.29) .. controls (6.95,-1.4) and (3.31,-0.3) .. (0,0) .. controls (3.31,0.3) and (6.95,1.4) .. (10.93,3.29)   ;
	\draw    (348.84,242.33) -- (348.84,212.17) ;
	\draw    (210.49,100.43) -- (210.49,114.44) ;
	\draw [shift={(210.49,114.44)}, rotate = 90] [color={rgb, 255:red, 0; green, 0; blue, 0 }  ][line width=0.75]    (10.93,-3.29) .. controls (6.95,-1.4) and (3.31,-0.3) .. (0,0) .. controls (3.31,0.3) and (6.95,1.4) .. (10.93,3.29)   ;
	\draw [shift={(210.49,100.43)}, rotate = 90] [color={rgb, 255:red, 0; green, 0; blue, 0 }  ][fill={rgb, 255:red, 0; green, 0; blue, 0 }  ][line width=0.75]      (0, 0) circle [x radius= 3.35, y radius= 3.35]   ;
	\draw    (348.84,182.01) -- (348.84,210.17) ;
	\draw [shift={(348.84,210.17)}, rotate = 90] [color={rgb, 255:red, 0; green, 0; blue, 0 }  ][line width=0.75]    (10.93,-3.29) .. controls (6.95,-1.4) and (3.31,-0.3) .. (0,0) .. controls (3.31,0.3) and (6.95,1.4) .. (10.93,3.29)   ;
	\draw [shift={(348.84,182.01)}, rotate = 90] [color={rgb, 255:red, 0; green, 0; blue, 0 }  ][fill={rgb, 255:red, 0; green, 0; blue, 0 }  ][line width=0.75]      (0, 0) circle [x radius= 3.35, y radius= 3.35]   ;
	\draw [color={rgb, 255:red, 0; green, 0; blue, 0 }  ,draw opacity=1 ][line width=0.75]    (349.18,80.33) -- (349.16,45.22) ;
	\draw [color={rgb, 255:red, 0; green, 0; blue, 0 }  ,draw opacity=1 ][line width=0.75]    (396.33,100.34) -- (348.84,100.43) ;
	\draw [color={rgb, 255:red, 0; green, 0; blue, 0 }  ,draw opacity=1 ][line width=0.75]    (348.84,182.01) -- (301.34,182.1) ;
	\draw [color={rgb, 255:red, 208; green, 2; blue, 27 }  ,draw opacity=1 ][line width=1.5]    (251.27,100.34) -- (224.08,100.43) ;
	\draw [color={rgb, 255:red, 74; green, 144; blue, 226 }  ,draw opacity=1 ][line width=1.5]    (251.27,100.43) -- (251.6,80.33) ;
	\draw   (156.11,100.43) .. controls (156.11,71.11) and (179.88,47.34) .. (209.2,47.34) .. controls (238.53,47.34) and (262.3,71.11) .. (262.3,100.43) .. controls (262.3,129.76) and (238.53,153.53) .. (209.2,153.53) .. controls (179.88,153.53) and (156.11,129.76) .. (156.11,100.43) -- cycle ;
	\draw   (295.74,182.01) .. controls (295.74,152.69) and (319.51,128.92) .. (348.84,128.92) .. controls (378.16,128.92) and (401.93,152.69) .. (401.93,182.01) .. controls (401.93,211.33) and (378.16,235.1) .. (348.84,235.1) .. controls (319.51,235.1) and (295.74,211.33) .. (295.74,182.01) -- cycle ;
	\draw [color={rgb, 255:red, 208; green, 2; blue, 27 }  ,draw opacity=1 ][line width=1.5]    (348.84,100.43) -- (321.65,100.53) ;
	\draw [color={rgb, 255:red, 74; green, 144; blue, 226 }  ,draw opacity=1 ][line width=1.5]    (348.84,100.43) -- (349.18,80.33) ;
	\draw [color={rgb, 255:red, 208; green, 2; blue, 27 }  ,draw opacity=1 ][line width=1.5]    (348.84,166.93) -- (321.65,167.02) ;
	\draw [color={rgb, 255:red, 74; green, 144; blue, 226 }  ,draw opacity=1 ][line width=1.5]    (348.84,166.93) -- (349.18,146.83) ;
	\draw [color={rgb, 255:red, 128; green, 128; blue, 128 }  ,draw opacity=1 ]   (236.52,115.06) .. controls (259.12,181.77) and (273.26,121.28) .. (315.83,149) ;
	\draw [shift={(235.83,113)}, rotate = 71.81] [color={rgb, 255:red, 128; green, 128; blue, 128 }  ,draw opacity=1 ][line width=0.75]    (10.93,-3.29) .. controls (6.95,-1.4) and (3.31,-0.3) .. (0,0) .. controls (3.31,0.3) and (6.95,1.4) .. (10.93,3.29)   ;
	\draw    (224.42,80.33) -- (251.6,80.33) ;
	\draw    (321.99,80.33) -- (349.18,80.33) ;
	\draw    (321.99,146.83) -- (349.18,146.83) ;
	\draw    (224.42,80.33) -- (224.08,100.43) ;
	\draw    (321.99,80.43) -- (321.65,100.53) ;
	\draw    (321.99,146.92) -- (321.65,167.02) ;
	
	\draw (302.67,152.91) node [anchor=north west][inner sep=0.75pt]  [font=\scriptsize,color={rgb, 255:red, 0; green, 0; blue, 0 }  ,opacity=1 ]  {$H_{i}^{2}$};
	\draw (329.67,131.41) node [anchor=north west][inner sep=0.75pt]  [font=\scriptsize,color={rgb, 255:red, 0; green, 0; blue, 0 }  ,opacity=1 ]  {$V_{i}^{2}$};
	\draw (140.25,34.9) node [anchor=north west][inner sep=0.75pt]    {$U_{i-1}$};
	\draw (403.75,145.9) node [anchor=north west][inner sep=0.75pt]    {$U_{i}$};
	\draw (251,151.4) node [anchor=north west][inner sep=0.75pt]  [font=\footnotesize,color={rgb, 255:red, 155; green, 155; blue, 155 }  ,opacity=1 ]  {$f^{-m_{i}}$};
	\draw (220,60.07) node [anchor=north west][inner sep=0.75pt]    {$Z_{i}$};

\end{tikzpicture}

\end{center}

The choice of $Q_i$ and the position of $H^1_i$ and $V^1_i$ was necessary to assure that the forward iterates of $Y_i$ intersect $Y_{i+1}$.

This process defines $Y_i$ from $i=1$ up to $i=L(f)-1$, and we finish it by defining $Y_{L(f)}$ as a rectangle-like compact neighborhood in $Q_{L(f)-1}$, for which one of the curves that define its border is contained in $W^u(p_{L(f)-1})$. For this case, $Z_{L(f)}=Y_{L(f)}$ and $m_{L(f)}=0$. 

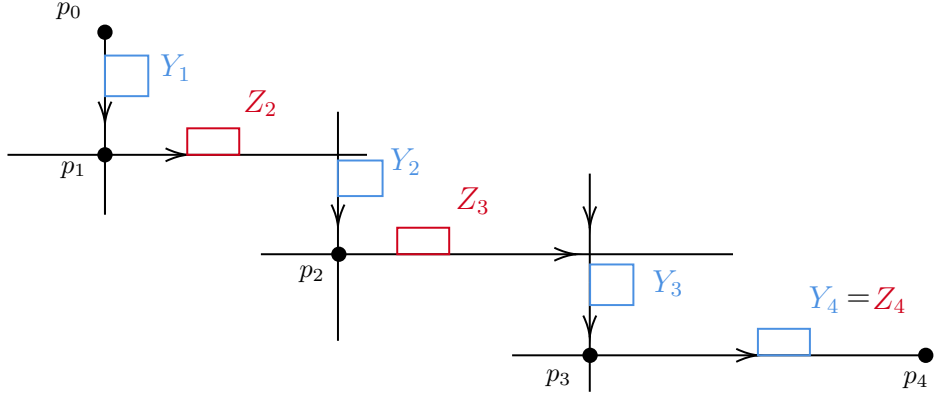
\begin{center}

\tikzset{every picture/.style={line width=0.75pt}} 

\begin{tikzpicture}[x=0.75pt,y=0.75pt,yscale=-1,xscale=1]
	
	\draw    (11.33,86.13) -- (192.65,86.13) ;
	\draw [shift={(101.99,86.13)}, rotate = 180] [color={rgb, 255:red, 0; green, 0; blue, 0 }  ][line width=0.75]    (10.93,-3.29) .. controls (6.95,-1.4) and (3.31,-0.3) .. (0,0) .. controls (3.31,0.3) and (6.95,1.4) .. (10.93,3.29)   ;
	\draw    (60.49,24.24) -- (60.49,116.24) ;
	\draw [shift={(60.49,70.24)}, rotate = 270] [color={rgb, 255:red, 0; green, 0; blue, 0 }  ][line width=0.75]    (10.93,-3.29) .. controls (6.95,-1.4) and (3.31,-0.3) .. (0,0) .. controls (3.31,0.3) and (6.95,1.4) .. (10.93,3.29)   ;
	\draw [shift={(60.49,24.24)}, rotate = 90] [color={rgb, 255:red, 0; green, 0; blue, 0 }  ][fill={rgb, 255:red, 0; green, 0; blue, 0 }  ][line width=0.75]      (0, 0) circle [x radius= 3.35, y radius= 3.35]   ;
	\draw    (60.49,86.13) -- (60.49,116.24) ;
	\draw [shift={(60.49,86.13)}, rotate = 90] [color={rgb, 255:red, 0; green, 0; blue, 0 }  ][fill={rgb, 255:red, 0; green, 0; blue, 0 }  ][line width=0.75]      (0, 0) circle [x radius= 3.35, y radius= 3.35]   ;
	\draw    (178,64.39) -- (178,180) ;
	\draw [shift={(178,122.19)}, rotate = 270] [color={rgb, 255:red, 0; green, 0; blue, 0 }  ][line width=0.75]    (10.93,-3.29) .. controls (6.95,-1.4) and (3.31,-0.3) .. (0,0) .. controls (3.31,0.3) and (6.95,1.4) .. (10.93,3.29)   ;
	\draw    (178.45,136.31) -- (218.86,136.31) ;
	\draw    (305,150.67) -- (305,205.73) ;
	\draw [shift={(305,178.2)}, rotate = 270] [color={rgb, 255:red, 0; green, 0; blue, 0 }  ][line width=0.75]    (10.93,-3.29) .. controls (6.95,-1.4) and (3.31,-0.3) .. (0,0) .. controls (3.31,0.3) and (6.95,1.4) .. (10.93,3.29)   ;
	\draw    (178.45,136.31) -- (139.13,136.31) ;
	\draw [shift={(178.45,136.31)}, rotate = 180] [color={rgb, 255:red, 0; green, 0; blue, 0 }  ][fill={rgb, 255:red, 0; green, 0; blue, 0 }  ][line width=0.75]      (0, 0) circle [x radius= 3.35, y radius= 3.35]   ;
	\draw    (305.15,187.33) -- (474.46,187.33) ;
	\draw [shift={(474.46,187.33)}, rotate = 0] [color={rgb, 255:red, 0; green, 0; blue, 0 }  ][fill={rgb, 255:red, 0; green, 0; blue, 0 }  ][line width=0.75]      (0, 0) circle [x radius= 3.35, y radius= 3.35]   ;
	\draw [shift={(389.8,187.33)}, rotate = 180] [color={rgb, 255:red, 0; green, 0; blue, 0 }  ][line width=0.75]    (10.93,-3.29) .. controls (6.95,-1.4) and (3.31,-0.3) .. (0,0) .. controls (3.31,0.3) and (6.95,1.4) .. (10.93,3.29)   ;
	\draw    (305.15,187.33) -- (265.83,187.33) ;
	\draw [shift={(305.15,187.33)}, rotate = 180] [color={rgb, 255:red, 0; green, 0; blue, 0 }  ][fill={rgb, 255:red, 0; green, 0; blue, 0 }  ][line width=0.75]      (0, 0) circle [x radius= 3.35, y radius= 3.35]   ;
	\draw  [color={rgb, 255:red, 74; green, 144; blue, 226 }  ,draw opacity=1 ] (60.49,36.03) -- (82.33,36.03) -- (82.33,56.52) -- (60.49,56.52) -- cycle ;
	\draw  [color={rgb, 255:red, 208; green, 2; blue, 27 }  ,draw opacity=1 ] (101.99,72.75) -- (128.21,72.75) -- (128.21,86.13) -- (101.99,86.13) -- cycle ;
	\draw  [color={rgb, 255:red, 74; green, 144; blue, 226 }  ,draw opacity=1 ] (389.8,173.95) -- (416.02,173.95) -- (416.02,187.33) -- (389.8,187.33) -- cycle ;
	\draw    (218.86,136.31) -- (377.24,136.31) ;
	\draw [shift={(298.05,136.31)}, rotate = 180] [color={rgb, 255:red, 0; green, 0; blue, 0 }  ][line width=0.75]    (10.93,-3.29) .. controls (6.95,-1.4) and (3.31,-0.3) .. (0,0) .. controls (3.31,0.3) and (6.95,1.4) .. (10.93,3.29)   ;
	\draw  [color={rgb, 255:red, 208; green, 2; blue, 27 }  ,draw opacity=1 ] (207.89,122.93) -- (234.1,122.93) -- (234.1,136.31) -- (207.89,136.31) -- cycle ;
	\draw  [color={rgb, 255:red, 74; green, 144; blue, 226 }  ,draw opacity=1 ] (177.95,89) -- (200.5,89) -- (200.5,106.95) -- (177.95,106.95) -- cycle ;
	\draw    (305,95.61) -- (305,150.67) ;
	\draw [shift={(305,123.14)}, rotate = 270] [color={rgb, 255:red, 0; green, 0; blue, 0 }  ][line width=0.75]    (10.93,-3.29) .. controls (6.95,-1.4) and (3.31,-0.3) .. (0,0) .. controls (3.31,0.3) and (6.95,1.4) .. (10.93,3.29)   ;
	\draw  [color={rgb, 255:red, 74; green, 144; blue, 226 }  ,draw opacity=1 ] (305.15,141.47) -- (327,141.47) -- (327,161.96) -- (305.15,161.96) -- cycle ;
	
	\draw (34.69,8.02) node [anchor=north west][inner sep=0.75pt]  [font=\footnotesize]  {$p_{0}$};
	\draw (36.87,87.89) node [anchor=north west][inner sep=0.75pt]  [font=\footnotesize]  {$p_{1}$};
	\draw (157.02,140.58) node [anchor=north west][inner sep=0.75pt]  [font=\footnotesize]  {$p_{2}$};
	\draw (281.54,193.27) node [anchor=north west][inner sep=0.75pt]  [font=\footnotesize]  {$p_{3}$};
	\draw (461.77,194.11) node [anchor=north west][inner sep=0.75pt]  [font=\footnotesize]  {$p_{4}$};
	\draw (86.85,34.04) node [anchor=north west][inner sep=0.75pt]  [color={rgb, 255:red, 74; green, 144; blue, 226 }  ,opacity=1 ]  {$Y_{1}$};
	\draw (202.91,81.4) node [anchor=north west][inner sep=0.75pt]  [color={rgb, 255:red, 74; green, 144; blue, 226 }  ,opacity=1 ]  {$Y_{2}$};
	\draw (335.16,142.77) node [anchor=north west][inner sep=0.75pt]  [color={rgb, 255:red, 74; green, 144; blue, 226 }  ,opacity=1 ]  {$Y_{3}$};
	\draw (413.99,151.11) node [anchor=north west][inner sep=0.75pt]  [color={rgb, 255:red, 74; green, 144; blue, 226 }  ,opacity=1 ]  {$Y_{4}$};
	\draw (128.83,51.93) node [anchor=north west][inner sep=0.75pt]  [color={rgb, 255:red, 208; green, 2; blue, 27 }  ,opacity=1 ]  {$Z_{2}$};
	\draw (235.89,100.95) node [anchor=north west][inner sep=0.75pt]  [color={rgb, 255:red, 208; green, 2; blue, 27 }  ,opacity=1 ]  {$Z_{3}$};
	\draw (446.16,151.27) node [anchor=north west][inner sep=0.75pt]  [color={rgb, 255:red, 208; green, 2; blue, 27 }  ,opacity=1 ]  {$Z_{4}$};
	\draw (433,155) node [anchor=north west][inner sep=0.75pt]    {$=$};

\end{tikzpicture}

\end{center}

We say that a compact neighborhood $B\subset Y_i$ is a horizontal strip if its border is formed by $4$ curves such that two of them are contained in $V^1_i$ and $V^2_i$. Then, we define a vertical strip analogously. It is simple to observe that in the same $Y_i$, any horizontal strip intersects any vertical strip in a non-empty compact neighborhood. 

Now, we proceed to study the transition maps between $Y_i$ and $Z_{i+1}$. For each $i$, we consider the set $D_i=\{x\in Y_i:\exists n\geq 0\ /f^n(x)\in Z_{i+1}\}$ and define the map $T_i:D_i\subset Y_i \to Z_{i+1}$ such that $T_i(x) = f^n(x)$, where $n$ is such that $f^n(x)\in Z_{i+1}$. For a fixed $n\in\N$, we call $B^i_n$ as the set of points in $D_i$ such that $T(x) = f^n(x)$. Recall that $Y_i$ and $Z_{i+1}$ are subsets in $U_i$, which is a linearizing neighborhood of $p_i$. From this, it is simple to observe that there exists some $l_i$ such that $B^i_n$ is a vertical strip in $Y_i$ and $T_i(B^i_n)$ is a horizontal strip in $Z_{i+1}$ for all $n\geq l_i$.

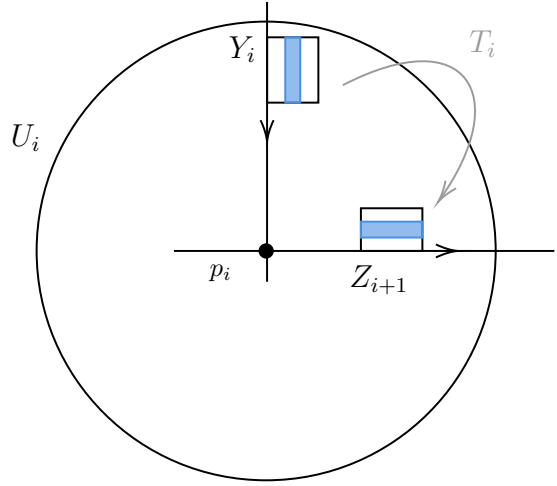
\begin{center}

\tikzset{every picture/.style={line width=0.75pt}} 

\begin{tikzpicture}[x=0.75pt,y=0.75pt,yscale=-1,xscale=1]
	
	\draw    (284.78,44) -- (284.78,185.23) ;
	\draw [shift={(284.78,114.62)}, rotate = 270] [color={rgb, 255:red, 0; green, 0; blue, 0 }  ][line width=0.75]    (10.93,-3.29) .. controls (6.95,-1.4) and (3.31,-0.3) .. (0,0) .. controls (3.31,0.3) and (6.95,1.4) .. (10.93,3.29)   ;
	\draw    (284.32,169.68) -- (332.11,169.68) ;
	\draw    (284.32,169.68) -- (237.81,169.68) ;
	\draw [shift={(284.32,169.68)}, rotate = 180] [color={rgb, 255:red, 0; green, 0; blue, 0 }  ][fill={rgb, 255:red, 0; green, 0; blue, 0 }  ][line width=0.75]      (0, 0) circle [x radius= 3.35, y radius= 3.35]   ;
	\draw  [color={rgb, 255:red, 0; green, 0; blue, 0 }  ,draw opacity=1 ] (284.82,61.85) -- (310.65,61.85) -- (310.65,94.83) -- (284.82,94.83) -- cycle ;
	\draw    (332.11,169.68) -- (429.67,169.68) ;
	\draw [shift={(380.89,169.68)}, rotate = 180] [color={rgb, 255:red, 0; green, 0; blue, 0 }  ][line width=0.75]    (10.93,-3.29) .. controls (6.95,-1.4) and (3.31,-0.3) .. (0,0) .. controls (3.31,0.3) and (6.95,1.4) .. (10.93,3.29)   ;
	\draw  [color={rgb, 255:red, 0; green, 0; blue, 0 }  ,draw opacity=1 ] (332.11,148.14) -- (363.12,148.14) -- (363.12,169.68) -- (332.11,169.68) -- cycle ;
	\draw   (168.56,169.68) .. controls (168.56,105.75) and (220.39,53.92) .. (284.32,53.92) .. controls (348.25,53.92) and (400.08,105.75) .. (400.08,169.68) .. controls (400.08,233.61) and (348.25,285.44) .. (284.32,285.44) .. controls (220.39,285.44) and (168.56,233.61) .. (168.56,169.68) -- cycle ;
	\draw  [color={rgb, 255:red, 74; green, 144; blue, 226 }  ,draw opacity=1 ][fill={rgb, 255:red, 74; green, 144; blue, 226 }  ,fill opacity=0.6 ] (293.96,61.85) -- (301.51,61.85) -- (301.51,94.83) -- (293.96,94.83) -- cycle ;
	\draw  [color={rgb, 255:red, 74; green, 144; blue, 226 }  ,draw opacity=1 ][fill={rgb, 255:red, 74; green, 144; blue, 226 }  ,fill opacity=0.6 ] (332.11,154.9) -- (363.12,154.9) -- (363.12,162.92) -- (332.11,162.92) -- cycle ;
	\draw [color={rgb, 255:red, 155; green, 155; blue, 155 }  ,draw opacity=1 ]   (323,86) .. controls (386.36,53.33) and (408.56,91.88) .. (372.12,143.43) ;
	\draw [shift={(371,145)}, rotate = 305.98] [color={rgb, 255:red, 155; green, 155; blue, 155 }  ,draw opacity=1 ][line width=0.75]    (10.93,-3.29) .. controls (6.95,-1.4) and (3.31,-0.3) .. (0,0) .. controls (3.31,0.3) and (6.95,1.4) .. (10.93,3.29)   ;
	
	\draw (254.35,175.25) node [anchor=north west][inner sep=0.75pt]  [font=\footnotesize]  {$p_{i}$};
	\draw (264.04,59.64) node [anchor=north west][inner sep=0.75pt]  [color={rgb, 255:red, 0; green, 0; blue, 0 }  ,opacity=1 ]  {$Y_{i}{}$};
	\draw (324.83,175.33) node [anchor=north west][inner sep=0.75pt]  [color={rgb, 255:red, 0; green, 0; blue, 0 }  ,opacity=1 ]  {$Z_{i+1}$};
	\draw (385.33,56.73) node [anchor=north west][inner sep=0.75pt]  [color={rgb, 255:red, 155; green, 155; blue, 155 }  ,opacity=1 ]  {$T_{i}$};
	\draw (154,104.4) node [anchor=north west][inner sep=0.75pt]    {$U_{i}$};

\end{tikzpicture}

\end{center}

Next, we fix $k_i = l_i+m_{i+1}$ and for each  $n_i\geq k_i$ for $i=1,\cdots, L(f)-1$. We consider the set 
\[ A(n_1,\cdots, n_{L(f)-1})=\{x\in Y_1: f^{n_1}(x) \in Y_2, f^{n_1+n_2}(x) \in Y_3,\cdots,  f^{n_1+\cdots + n_{L(f)-1}}(x)\in Y_{L(f)}\}.\] 

If we prove that for any tuple $(n_1,\cdots, n_{L(f)-1})$, the set $A(n_1,\cdots, n_{L(f)-1})$ is not empty, we have proved that $\F=\{Y_1,\cdots, Y_{L(f)}\}$ is combinatorially complete, and from this, the theorem follows.

Let us fix the intermediary sets as follows: 
\[\begin{array}{rcl}
	A(n_1)&=&\{x\in Y_1: f^{n_1}(x) \in Y_2\}  \\
	A(n_1,n_2)&=&\{x\in A(n_1): f^{n_1+n_2}(x) \in Y_3\} \\
	&\vdots &
	\\
	A(n_1,\cdots, n_{L(f)-2})&=& \{x\in A(n_1,\cdots, n_{L(f)-3}): f^{n_1+\cdots + n_{L(f)-2}}(x)\in Y_{L(f)-1}\} \\
\end{array}\]

We prove by induction that each aforementioned set is a non-trivial vertical strip in $Y_1$ ordered by inclusion. By our definition, $A(n_1) = B^1_{n_1-m_2}$ because $f^{n_1}(B^1_{n_1-m_2}) = f^{m_2}(f^{n_1 - m_2}( B^1_{n_1-m_2}))\subset f^{m_2}(Z_2)\subset Y_2$. In addition, note that $f^{n_1}(A(n_1))$ is a horizontal strip in $Y_2$. 

\begin{center}

\tikzset{every picture/.style={line width=0.75pt}} 

\begin{tikzpicture}[x=0.75pt,y=0.75pt,yscale=-1,xscale=1]
	
	\draw    (31.33,150.08) -- (288.34,150.08) ;
	\draw [shift={(159.84,150.08)}, rotate = 180] [color={rgb, 255:red, 0; green, 0; blue, 0 }  ][line width=0.75]    (10.93,-3.29) .. controls (6.95,-1.4) and (3.31,-0.3) .. (0,0) .. controls (3.31,0.3) and (6.95,1.4) .. (10.93,3.29)   ;
	\draw    (101,66.23) -- (101,150.08) ;
	\draw [shift={(101,150.08)}, rotate = 90] [color={rgb, 255:red, 0; green, 0; blue, 0 }  ][fill={rgb, 255:red, 0; green, 0; blue, 0 }  ][line width=0.75]      (0, 0) circle [x radius= 3.35, y radius= 3.35]   ;
	\draw [shift={(101,108.15)}, rotate = 270] [color={rgb, 255:red, 0; green, 0; blue, 0 }  ][line width=0.75]    (10.93,-3.29) .. controls (6.95,-1.4) and (3.31,-0.3) .. (0,0) .. controls (3.31,0.3) and (6.95,1.4) .. (10.93,3.29)   ;
	\draw [shift={(101,66.23)}, rotate = 90] [color={rgb, 255:red, 0; green, 0; blue, 0 }  ][fill={rgb, 255:red, 0; green, 0; blue, 0 }  ][line width=0.75]      (0, 0) circle [x radius= 3.35, y radius= 3.35]   ;
	\draw    (101,150.08) -- (101,190.88) ;
	\draw    (268.22,168.9) -- (268.22,237.34) ;
	\draw [shift={(268.22,203.12)}, rotate = 270] [color={rgb, 255:red, 0; green, 0; blue, 0 }  ][line width=0.75]    (10.93,-3.29) .. controls (6.95,-1.4) and (3.31,-0.3) .. (0,0) .. controls (3.31,0.3) and (6.95,1.4) .. (10.93,3.29)   ;
	\draw    (268.22,218.07) -- (325.5,218.07) ;
	\draw    (268.22,218.07) -- (212.48,218.07) ;
	\draw [shift={(268.22,218.07)}, rotate = 180] [color={rgb, 255:red, 0; green, 0; blue, 0 }  ][fill={rgb, 255:red, 0; green, 0; blue, 0 }  ][line width=0.75]      (0, 0) circle [x radius= 3.35, y radius= 3.35]   ;
	\draw  [color={rgb, 255:red, 0; green, 0; blue, 0 }  ,draw opacity=1 ] (101,80.39) -- (131.97,80.39) -- (131.97,108.15) -- (101,108.15) -- cycle ;
	\draw  [color={rgb, 255:red, 0; green, 0; blue, 0 }  ,draw opacity=1 ] (268.22,160.06) -- (299.18,160.06) -- (299.18,187.82) -- (268.22,187.82) -- cycle ;
	\draw  [color={rgb, 255:red, 0; green, 0; blue, 0 }  ,draw opacity=1 ] (159.84,124) -- (197,124) -- (197,150.08) -- (159.84,150.08) -- cycle ;
	\draw    (325.5,218.07) -- (425,218.07) ;
	\draw [shift={(375.25,218.07)}, rotate = 180] [color={rgb, 255:red, 0; green, 0; blue, 0 }  ][line width=0.75]    (10.93,-3.29) .. controls (6.95,-1.4) and (3.31,-0.3) .. (0,0) .. controls (3.31,0.3) and (6.95,1.4) .. (10.93,3.29)   ;
	\draw  [color={rgb, 255:red, 0; green, 0; blue, 0 }  ,draw opacity=1 ] (325.5,194) -- (362.66,194) -- (362.66,218.07) -- (325.5,218.07) -- cycle ;
	\draw  [color={rgb, 255:red, 208; green, 2; blue, 27 }  ,draw opacity=1 ][fill={rgb, 255:red, 208; green, 2; blue, 27 }  ,fill opacity=0.6 ] (111.24,80.39) -- (121.74,80.39) -- (121.74,108.15) -- (111.24,108.15) -- cycle ;
	\draw  [color={rgb, 255:red, 208; green, 2; blue, 27 }  ,draw opacity=1 ][fill={rgb, 255:red, 208; green, 2; blue, 27 }  ,fill opacity=0.6 ] (159.84,131.67) -- (197,131.67) -- (197,142.42) -- (159.84,142.42) -- cycle ;
	\draw  [color={rgb, 255:red, 74; green, 144; blue, 226 }  ,draw opacity=1 ][fill={rgb, 255:red, 74; green, 144; blue, 226 }  ,fill opacity=0.6 ] (279.38,160.06) -- (288.02,160.06) -- (288.02,187.82) -- (279.38,187.82) -- cycle ;
	\draw  [color={rgb, 255:red, 74; green, 144; blue, 226 }  ,draw opacity=1 ][fill={rgb, 255:red, 74; green, 144; blue, 226 }  ,fill opacity=0.6 ] (325.5,201.68) -- (362.66,201.68) -- (362.66,210.39) -- (325.5,210.39) -- cycle ;
	\draw [color={rgb, 255:red, 155; green, 155; blue, 155 }  ,draw opacity=1 ]   (135.67,74.67) .. controls (167.84,51.68) and (189.67,62.98) .. (190.32,101.55) ;
	\draw [shift={(190.33,103.33)}, rotate = 270] [color={rgb, 255:red, 155; green, 155; blue, 155 }  ,draw opacity=1 ][line width=0.75]    (10.93,-3.29) .. controls (6.95,-1.4) and (3.31,-0.3) .. (0,0) .. controls (3.31,0.3) and (6.95,1.4) .. (10.93,3.29)   ;
	\draw [color={rgb, 255:red, 155; green, 155; blue, 155 }  ,draw opacity=1 ]   (139,68) .. controls (234.52,14.93) and (299.02,77.37) .. (300.32,145.64) ;
	\draw [shift={(300.33,146.67)}, rotate = 269.44] [color={rgb, 255:red, 155; green, 155; blue, 155 }  ,draw opacity=1 ][line width=0.75]    (10.93,-3.29) .. controls (6.95,-1.4) and (3.31,-0.3) .. (0,0) .. controls (3.31,0.3) and (6.95,1.4) .. (10.93,3.29)   ;
	\draw [color={rgb, 255:red, 155; green, 155; blue, 155 }  ,draw opacity=1 ]   (315,158) .. controls (345.05,142.32) and (361.02,155.45) .. (358.51,182.34) ;
	\draw [shift={(358.33,184)}, rotate = 276.79] [color={rgb, 255:red, 155; green, 155; blue, 155 }  ,draw opacity=1 ][line width=0.75]    (10.93,-3.29) .. controls (6.95,-1.4) and (3.31,-0.3) .. (0,0) .. controls (3.31,0.3) and (6.95,1.4) .. (10.93,3.29)   ;
	\draw  [color={rgb, 255:red, 208; green, 2; blue, 27 }  ,draw opacity=1 ][fill={rgb, 255:red, 208; green, 2; blue, 27 }  ,fill opacity=0.6 ] (268.22,168.9) -- (299.18,168.9) -- (299.18,178.98) -- (268.22,178.98) -- cycle ;
	\draw    (268.22,100.47) -- (268.22,168.9) ;
	\draw [shift={(268.22,134.69)}, rotate = 270] [color={rgb, 255:red, 0; green, 0; blue, 0 }  ][line width=0.75]    (10.93,-3.29) .. controls (6.95,-1.4) and (3.31,-0.3) .. (0,0) .. controls (3.31,0.3) and (6.95,1.4) .. (10.93,3.29)   ;
	
	\draw (75.57,55.59) node [anchor=north west][inner sep=0.75pt]  [font=\footnotesize]  {$p_{0}$};
	\draw (78.67,154.81) node [anchor=north west][inner sep=0.75pt]  [font=\footnotesize]  {$p_{1}$};
	\draw (240.98,226.21) node [anchor=north west][inner sep=0.75pt]  [font=\footnotesize]  {$p_{2}$};
	\draw (111.43,56.22) node [anchor=north west][inner sep=0.75pt]  [color={rgb, 255:red, 0; green, 0; blue, 0 }  ,opacity=1 ]  {$Y_{1}$};
	\draw (301.18,172.3) node [anchor=north west][inner sep=0.75pt]  [color={rgb, 255:red, 0; green, 0; blue, 0 }  ,opacity=1 ]  {$Y_{2}$};
	\draw (202.27,106.8) node [anchor=north west][inner sep=0.75pt]  [color={rgb, 255:red, 0; green, 0; blue, 0 }  ,opacity=1 ]  {$Z_{2}$};
	\draw (367.35,183.21) node [anchor=north west][inner sep=0.75pt]  [color={rgb, 255:red, 0; green, 0; blue, 0 }  ,opacity=1 ]  {$Z_{3}$};
	\draw (192,70.07) node [anchor=north west][inner sep=0.75pt]  [font=\footnotesize,color={rgb, 255:red, 155; green, 155; blue, 155 }  ,opacity=1 ]  {$f^{n_{1} -m_{2}}$};
	\draw (291.33,80.73) node [anchor=north west][inner sep=0.75pt]  [font=\footnotesize,color={rgb, 255:red, 155; green, 155; blue, 155 }  ,opacity=1 ]  {$f^{m_{2}}\left( f^{n_{1} -m_{2}}\right)$};
	\draw (360,142.07) node [anchor=north west][inner sep=0.75pt]  [font=\footnotesize,color={rgb, 255:red, 155; green, 155; blue, 155 }  ,opacity=1 ]  {$T_{2}$};

\end{tikzpicture}

\end{center}

Our inductive hypothesis is that $ A(n_1,\cdots, n_{i})$ is a non-empty vertical strip in $Y_1$ and $f^{n_1+\cdots + n_i}(A(n_1,\cdots, n_{i}))$ is a horizontal strip in $Y_{i+1}$. Take $n_{i+1}\geq k_{i+1}$ and recall that $B^{i+1}_{n_{i+1}-m_{i+2}}$ is a vertical strip in $Y_{i+1}$. Then, $f^{n_1+\cdots + n_i}(A(n_1,\cdots, n_{i}))\cap B^{i+1}_{n_{i+1}-m_{i+2}}$ is not empty. As it is a vertical strip in $f^{n_1+\cdots + n_i}(A(n_1,\cdots, n_{i}))$, its pre-image by $f^{-(n_1+\cdots + n_i)}$ is a vertical strip in $Y_1$. This set is, in fact, $A(n_1,\cdots,n_{i+1})$. On the other hand, we also know that 
\[f^{n_1+\cdots + n_i}(A(n_1,\cdots, n_{i+1})) = f^{n_1+\cdots + n_i}(A(n_1,\cdots, n_{i}))\cap B^{i+1}_{n_{i+1}-m_{i+2}},\]
which is now a horizontal strip in $B^{i+1}_{n_{i+1}-m_{i+2}}$. Then, $f^{n_1+\cdots + n_i + n_{i+1}-m_{i+2}}(A(n_1,\cdots, n_{i+1}))$ is also a horizontal strip in $Z_{i+2}$. With this, $f^{n_1+\cdots + n_{i+1}}(A(n_1,\cdots, n_{i+1}))$ is a horizontal strip in $Y_{i+2}$.

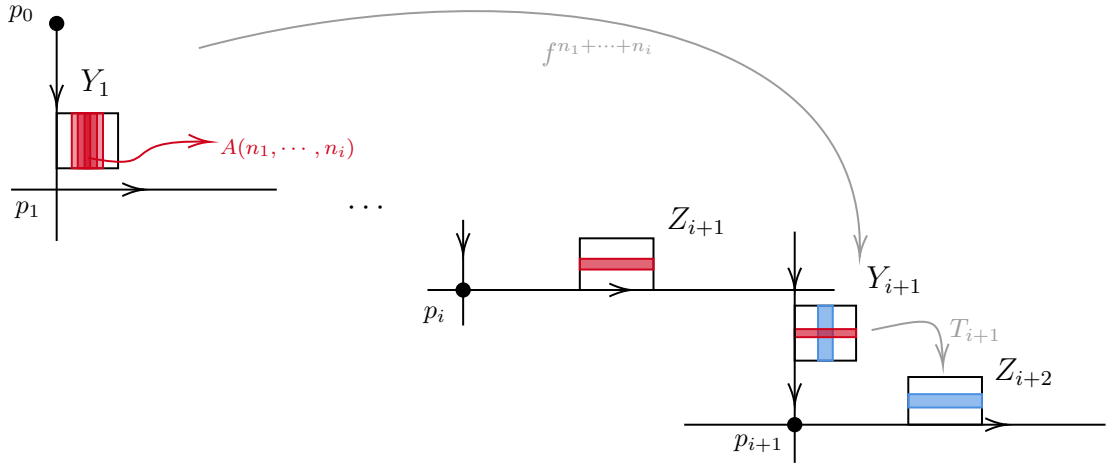
\begin{center}

\tikzset{every picture/.style={line width=0.75pt}} 

\begin{tikzpicture}[x=0.75pt,y=0.75pt,yscale=-1,xscale=1]
	
	\draw    (297,176.75) -- (502.34,176.75) ;
	\draw [shift={(399.67,176.75)}, rotate = 180] [color={rgb, 255:red, 0; green, 0; blue, 0 }  ][line width=0.75]    (10.93,-3.29) .. controls (6.95,-1.4) and (3.31,-0.3) .. (0,0) .. controls (3.31,0.3) and (6.95,1.4) .. (10.93,3.29)   ;
	\draw    (315,141.33) -- (315,176.75) ;
	\draw [shift={(315,176.75)}, rotate = 90] [color={rgb, 255:red, 0; green, 0; blue, 0 }  ][fill={rgb, 255:red, 0; green, 0; blue, 0 }  ][line width=0.75]      (0, 0) circle [x radius= 3.35, y radius= 3.35]   ;
	\draw [shift={(315,159.04)}, rotate = 270] [color={rgb, 255:red, 0; green, 0; blue, 0 }  ][line width=0.75]    (10.93,-3.29) .. controls (6.95,-1.4) and (3.31,-0.3) .. (0,0) .. controls (3.31,0.3) and (6.95,1.4) .. (10.93,3.29)   ;
	\draw    (315,176.75) -- (315,194.67) ;
	\draw    (482.22,205.65) -- (482.22,264) ;
	\draw [shift={(482.22,234.83)}, rotate = 270] [color={rgb, 255:red, 0; green, 0; blue, 0 }  ][line width=0.75]    (10.93,-3.29) .. controls (6.95,-1.4) and (3.31,-0.3) .. (0,0) .. controls (3.31,0.3) and (6.95,1.4) .. (10.93,3.29)   ;
	\draw    (482.22,244.74) -- (539.5,244.74) ;
	\draw    (482.22,244.74) -- (426.48,244.74) ;
	\draw [shift={(482.22,244.74)}, rotate = 180] [color={rgb, 255:red, 0; green, 0; blue, 0 }  ][fill={rgb, 255:red, 0; green, 0; blue, 0 }  ][line width=0.75]      (0, 0) circle [x radius= 3.35, y radius= 3.35]   ;
	\draw  [color={rgb, 255:red, 0; green, 0; blue, 0 }  ,draw opacity=1 ] (482.22,184.65) -- (513.18,184.65) -- (513.18,212.41) -- (482.22,212.41) -- cycle ;
	\draw  [color={rgb, 255:red, 0; green, 0; blue, 0 }  ,draw opacity=1 ] (373.84,150.67) -- (411,150.67) -- (411,176.75) -- (373.84,176.75) -- cycle ;
	\draw    (539.5,244.74) -- (639,244.74) ;
	\draw [shift={(589.25,244.74)}, rotate = 180] [color={rgb, 255:red, 0; green, 0; blue, 0 }  ][line width=0.75]    (10.93,-3.29) .. controls (6.95,-1.4) and (3.31,-0.3) .. (0,0) .. controls (3.31,0.3) and (6.95,1.4) .. (10.93,3.29)   ;
	\draw  [color={rgb, 255:red, 0; green, 0; blue, 0 }  ,draw opacity=1 ] (539.5,220.67) -- (576.66,220.67) -- (576.66,244.74) -- (539.5,244.74) -- cycle ;
	\draw    (87,126.08) -- (221,126.08) ;
	\draw [shift={(154,126.08)}, rotate = 180] [color={rgb, 255:red, 0; green, 0; blue, 0 }  ][line width=0.75]    (10.93,-3.29) .. controls (6.95,-1.4) and (3.31,-0.3) .. (0,0) .. controls (3.31,0.3) and (6.95,1.4) .. (10.93,3.29)   ;
	\draw    (110,42.23) -- (110,126.08) ;
	\draw [shift={(110,84.15)}, rotate = 270] [color={rgb, 255:red, 0; green, 0; blue, 0 }  ][line width=0.75]    (10.93,-3.29) .. controls (6.95,-1.4) and (3.31,-0.3) .. (0,0) .. controls (3.31,0.3) and (6.95,1.4) .. (10.93,3.29)   ;
	\draw [shift={(110,42.23)}, rotate = 90] [color={rgb, 255:red, 0; green, 0; blue, 0 }  ][fill={rgb, 255:red, 0; green, 0; blue, 0 }  ][line width=0.75]      (0, 0) circle [x radius= 3.35, y radius= 3.35]   ;
	\draw    (110,126.08) -- (110,152) ;
	\draw  [color={rgb, 255:red, 0; green, 0; blue, 0 }  ,draw opacity=1 ] (110,87.55) -- (140.97,87.55) -- (140.97,115.32) -- (110,115.32) -- cycle ;
	\draw  [color={rgb, 255:red, 208; green, 2; blue, 27 }  ,draw opacity=1 ][fill={rgb, 255:red, 208; green, 2; blue, 27 }  ,fill opacity=0.4 ] (117.61,87.55) -- (133.37,87.55) -- (133.37,115.32) -- (117.61,115.32) -- cycle ;
	\draw  [color={rgb, 255:red, 208; green, 2; blue, 27 }  ,draw opacity=1 ][fill={rgb, 255:red, 208; green, 2; blue, 27 }  ,fill opacity=0.5 ] (120.63,87.55) -- (130.34,87.55) -- (130.34,115.32) -- (120.63,115.32) -- cycle ;
	\draw  [color={rgb, 255:red, 208; green, 2; blue, 27 }  ,draw opacity=1 ][fill={rgb, 255:red, 208; green, 2; blue, 27 }  ,fill opacity=0.6 ] (124.06,87.55) -- (126.92,87.55) -- (126.92,115.32) -- (124.06,115.32) -- cycle ;
	\draw [color={rgb, 255:red, 208; green, 2; blue, 27 }  ,draw opacity=1 ]   (125.67,110) .. controls (166.26,117.92) and (137.91,100.36) .. (185.86,101.95) ;
	\draw [shift={(187.33,102)}, rotate = 182.29] [color={rgb, 255:red, 208; green, 2; blue, 27 }  ,draw opacity=1 ][line width=0.75]    (10.93,-3.29) .. controls (6.95,-1.4) and (3.31,-0.3) .. (0,0) .. controls (3.31,0.3) and (6.95,1.4) .. (10.93,3.29)   ;
	\draw  [color={rgb, 255:red, 208; green, 2; blue, 27 }  ,draw opacity=1 ][fill={rgb, 255:red, 208; green, 2; blue, 27 }  ,fill opacity=0.6 ] (373.84,161) -- (411,161) -- (411,166.42) -- (373.84,166.42) -- cycle ;
	\draw [color={rgb, 255:red, 155; green, 155; blue, 155 }  ,draw opacity=1 ]   (181.67,54.67) .. controls (323.95,16.19) and (518.05,30.85) .. (515.06,158.07) ;
	\draw [shift={(515,160)}, rotate = 272.22] [color={rgb, 255:red, 155; green, 155; blue, 155 }  ,draw opacity=1 ][line width=0.75]    (10.93,-3.29) .. controls (6.95,-1.4) and (3.31,-0.3) .. (0,0) .. controls (3.31,0.3) and (6.95,1.4) .. (10.93,3.29)   ;
	\draw  [color={rgb, 255:red, 74; green, 144; blue, 226 }  ,draw opacity=1 ][fill={rgb, 255:red, 74; green, 144; blue, 226 }  ,fill opacity=0.6 ] (493.97,184.65) -- (501.42,184.65) -- (501.42,212.41) -- (493.97,212.41) -- cycle ;
	\draw  [color={rgb, 255:red, 74; green, 144; blue, 226 }  ,draw opacity=1 ][fill={rgb, 255:red, 74; green, 144; blue, 226 }  ,fill opacity=0.6 ] (539.5,229.33) -- (576.66,229.33) -- (576.66,236.07) -- (539.5,236.07) -- cycle ;
	\draw  [color={rgb, 255:red, 208; green, 2; blue, 27 }  ,draw opacity=1 ][fill={rgb, 255:red, 208; green, 2; blue, 27 }  ,fill opacity=0.6 ] (482.22,196.49) -- (513.18,196.49) -- (513.18,200.56) -- (482.22,200.56) -- cycle ;
	\draw [color={rgb, 255:red, 155; green, 155; blue, 155 }  ,draw opacity=1 ]   (521,197) .. controls (553.18,190.18) and (555.88,189.05) .. (556.92,213.1) ;
	\draw [shift={(557,215)}, rotate = 267.8] [color={rgb, 255:red, 155; green, 155; blue, 155 }  ,draw opacity=1 ][line width=0.75]    (10.93,-3.29) .. controls (6.95,-1.4) and (3.31,-0.3) .. (0,0) .. controls (3.31,0.3) and (6.95,1.4) .. (10.93,3.29)   ;
	\draw    (482.22,147.29) -- (482.22,205.65) ;
	\draw [shift={(482.22,176.47)}, rotate = 270] [color={rgb, 255:red, 0; green, 0; blue, 0 }  ][line width=0.75]    (10.93,-3.29) .. controls (6.95,-1.4) and (3.31,-0.3) .. (0,0) .. controls (3.31,0.3) and (6.95,1.4) .. (10.93,3.29)   ;
	
	\draw (450.31,248.21) node [anchor=north west][inner sep=0.75pt]  [font=\footnotesize]  {$p_{i+1}$};
	\draw (517,163.4) node [anchor=north west][inner sep=0.75pt]  [color={rgb, 255:red, 0; green, 0; blue, 0 }  ,opacity=1 ]  {$Y_{i+1}$};
	\draw (416.27,133.46) node [anchor=north west][inner sep=0.75pt]  [color={rgb, 255:red, 0; green, 0; blue, 0 }  ,opacity=1 ]  {$Z_{i+1}$};
	\draw (581.35,209.88) node [anchor=north west][inner sep=0.75pt]  [color={rgb, 255:red, 0; green, 0; blue, 0 }  ,opacity=1 ]  {$Z_{i+2}$};
	\draw (84.57,31.59) node [anchor=north west][inner sep=0.75pt]  [font=\footnotesize]  {$p_{0}$};
	\draw (87.67,130.81) node [anchor=north west][inner sep=0.75pt]  [font=\footnotesize]  {$p_{1}$};
	\draw (120.43,63.89) node [anchor=north west][inner sep=0.75pt]  [color={rgb, 255:red, 0; green, 0; blue, 0 }  ,opacity=1 ]  {$Y_{1}$};
	\draw (255,130.4) node [anchor=north west][inner sep=0.75pt]    {$\cdots $};
	\draw (191,98.4) node [anchor=north west][inner sep=0.75pt]  [font=\tiny,color={rgb, 255:red, 208; green, 2; blue, 27 }  ,opacity=1 ]  {$A( n_{1} ,\cdots ,n_{i})$};
	\draw (353,49.07) node [anchor=north west][inner sep=0.75pt]  [font=\footnotesize,color={rgb, 255:red, 155; green, 155; blue, 155 }  ,opacity=1 ]  {$f^{n_{1} +\cdots +n_{i}}$};
	\draw (293.67,183.81) node [anchor=north west][inner sep=0.75pt]  [font=\footnotesize]  {$p_{i}$};
	\draw (559,190.07) node [anchor=north west][inner sep=0.75pt]  [font=\footnotesize,color={rgb, 255:red, 155; green, 155; blue, 155 }  ,opacity=1 ]  {$T_{i+1}$};

\end{tikzpicture}

\end{center}

 This brings us to the end of the proof of claim and, with it, the proof of theorem \ref{TeoMS}.
\end{proof}

We end this section with the proof of corollaries \ref{CorRig}, \ref{CorKS},  \ref{CorCont} and \ref{CorDimN}.

\begin{proof}[Proof of corollary \ref{CorRig}]
If $f:S\to S$ is such that $o(f)=[n]$, then $L(f)=1$. Since $f$ is Morse-Smale, it must always have at least one source and one sink, and it cannot have a saddle point; otherwise $L(f)\geq 2$. Note that for any sink, the stable manifold is an open set homeomorphic to $\R^2$ and the stable manifold of any source is only the source. Since the stable manifolds of the periodic points form a partition of the surface $S$ (that is connected), $S$ would not be a compact manifold if we had more than one source and one sink. Further, because we have one source and one sink, $S$ is the union of an open set homeomorphic to $\R^2$ and a point, and therefore, $S=S^2$ and $f$ has a North-South dynamic naturally.    
\end{proof}

\begin{proof}[Proof of corollary \ref{CorKS}]
Let us consider the families of disks $\mathcal{D}^m=\{D^m_1,\cdots, D^m_{l_m}\}$ as in subsection \ref{SubSecKS} and define the open sets $V^m$  associated to said families. We fix an $\e$ and  take an $m$ that is big enough such that $diam(D^{m}_i)\leq \e$ for all $D^m_i\in \mathcal{D}^m$. Let us take $K^m$ as the closure of the complement of $V^m$ and $f_m = f_{|K^m}$. Since the open sets $D^{m}_i$ are periodic, given any $n$, we can cover $V^m$ with $\#\mathcal{D}^m$  $(n,\e)$-balls. This means that $V^m$ adds to $g_{f,\e}(n)$ only a constant that does not depend on $n$. Therefore, $[g_{f,\epsilon}(n)]= [g_{f,\e,K^m}(n)]$ and then,
\[o(f) = \sup \{[g_{f,\e}(n)]:\e >0\} =  \sup \{[g_{f,\e, K^{m(\e)}}(n)]:\e >0 \}\leq \sup \{o(f_m):m\geq 0 \}.\]
Since $o(f_m)\leq o(f)$ for all $k$, we conclude that
\[o(f) = \sup \{o(f_m):m\geq 0 \}.\]

Further, We would like to point out that as the Hausdorff limit of the periodic points is in the set $\cap_m \overline{V^m}$, the non-wandering set of $f$ in $K^m$ only consists of finite periodic points. As $f$ is Kupka-Smale, we can extend $f_m$ to a map $\hat f_m$ on disk $\D^2$ with a periodic sink. Clearly $o(f_m) = o(\hat f_m)$ and $\hat f_m$ can be constructed differentiable, thereby obtaining that $\hat f_m$ is a Morse-Smale embedding of the disk. In this case, $L(\hat f_m)$ grows to infinity with $m\in \N$, and by theorem \ref{TeoCod}, we conclude
\[o(f) = \sup \{o(f_m):m\geq 0 \} =  \sup \{o(\hat f_m):m\geq 0 \} =  \sup \{[n^{L(\hat f_m)}]:m\geq 1\} = \sup(\Pol).\]

Let us now show that the Kupka-Smale are dense in the class of infinitely re-normalizable. We need to show that we can realize the classical perturbations of Kupka's theorem without breaking the infinitely re-normalizable condition. We take $\e>0$ and define $P_1$ as the set of periodic points in $K^1$. Although $P_1$ may not be finite, the periods of said points are bounded. As there are no periodic points of $P_1$ in the border of $K^1$, we can perturb $f$, obtaining $f_1$ such that

\begin{enumerate}
\item the $C^r$ distance between $f$ and $f_1$ is smaller than $\e/2$.
\item the support of the perturbation lies in the interior of $K^1$. 
\item all periodic points in $K^1$ are hyperbolic. 
\item all intersections of stable and unstable manifolds between the periodic points in $K^1$ are transverse. 
\end{enumerate}

For the final property, we observe that to obtain the transverse condition, the perturbation can be done near the periodic points. The second property is key because the periodicity of the disks $D^m_i$ remains, and therefore, $f_1$ is also infinitely re-normalizable. 

Next, we construct a family of maps $f_m$ by inductionally verifying the following:
\begin{enumerate}
\item The periodic discs $D^{\tilde m}_i$ are periodic for $f_m$ for all $\tilde m$ and all $i$. 
\item The $C^r$ distance between $f_m$ and $f_{m+1}$ is smaller than $\e/2^{m+1}$.
\item The support of the perturbation from $f_m$ to $f_{m+1}$ is contained in $V^{m+1}\setminus V^m$.
\item All the periodic points in $K^m$ are hyperbolic, and the intersections between their stable and unstable manifolds are transverse. 
\end{enumerate}

The construction is straightforward. The key elements of this construction are that there are no periodic points in the border of $V^m$ and the periods are bounded and that to obtain transversality, we can perturb near the periodic points. In the inductive step, once we have all the periodic points in 
$V^{m+1}\setminus V^m$ hyperbolic, we perturb to obtain transversality not only among them but also with the ones in $K^m$. 

We finish by taking $g$ as the limit of $f_m$. The infinitely re-normalizable property is preserved, as the disks $D^m_i$ are also periodic for $g$. By our choice in the support of the perturbation, the periodic points of $g$ in $K^m$ are the periodic points of $f_m$ in $K^m$ and are, therefore, hyperbolic. Moreover, every intersection between a stable and unstable manifold is transverse.   

\end{proof}

\begin{proof}[Proof of corollary \ref{CorCont}]
Based on corollary \ref{CorKS}, we know that $o(f) = \sup(\Pol)$, and using theorem \ref{TeoMS}, we only need to show that $\lim_k L(f_k) = \infty$. Let us take $\mathcal{D}^m=\{D^m_1,\cdots, D^m_{l_m}\}$, $V^m$ and $K^m$ as before. Given $m$, there exists some integer $l_m$ such that the periodic points of periods smaller than $l_m$ belong to a neighborhood of $K_m$. The fact that $f$ is Kupka-Smale implies that these points are finite. If $f_k$ is close enough to $f$, all of these periodic points must have a continuation in $f_k$ and the intersection between the stable and unstable manifolds shall remain non-empty if it happened for $f$. Therefore, $L(f_k)$ is at least $L(f_{|K^m})$, which, as explained earlier, grows toward infinity. 
\end{proof}

\begin{proof}[Proof of corollary \ref{CorDimN}]
By lemma \ref{LemUppBou}, we know that for any $\F\in \Sigma$, if $L=\#\F$, then $[c_{f,\F}(n)]\leq [n^{L}]$. By lemma \ref{LemLoc}, we know that $L\leq L(f)$, and by theorem \ref{TeoCod}, we conclude. 
\end{proof}
 \section{Proof of theorem \ref{TeoCod}}\label{SecTeoCod}

The main objective of this section is to prove theorem \ref{TeoCod}. The proof is split in two inequalities. The first one, 
\[o(f)\geq \sup\{[c_{f,\F}(n)]\in \OG: \F\in \Sigma\},\]
is deduced from the following lemma.

\begin{lemma}\label{Lem1}
For any $\F\in \Sigma$, let $L=\#\F$. Then, there exists $\e$ such that
\[c_{f,\F}(n)\leq 2^L s_{f,\e}(n).\]
\end{lemma}

The proof of the second inequality, 
\[o(f)\leq \sup\{[c_{f,\F}(n)]\in \OG: \F\in \Sigma\},\]
is split in three steps. Each one is a separate lemma and we shall enunciate them as follows.  

\begin{lemma}\label{Lem2}
Given $\e>0$ and $\delta>0$, there exists a family $\F=\{Y_1,\cdots,Y_L\}$ of wandering compact neighborhoods  and a constant $B>0$ such that $diam(Y_i)\leq \delta$ for all $i$ and $s_{f,\e}(n)\leq B c_{f,\F}(n)$. Moreover, $\F$ can be chosen such that if $Y_i\cap Y_j\neq \emptyset$, then $Y_i\cup Y_j$ is also a wandering set. 
\end{lemma}

\begin{lemma}\label{Lem3}
Given a family $\F$ of wandering compact neighborhoods as in the previous lemma, the following equation holds.
\[[c_{f,\F}(n)]\leq \sup\{[c_{f,\F'}(n)]; \F'\subset \F \,\text{is disjoint}\}.\]
\end{lemma}

\begin{lemma}\label{Lem4}
Given a finite family $\F$ of disjoint wandering compact neighborhoods, the following equation holds. 
\[[c_{f,\F}(n)]\leq \sup\{[c_{f,\F'}(n)]; \F' \subset \F \text{ is mutually singular}\}.\]
\end{lemma}

We would like to comment that in the last two lemmas, lemmas \ref{Lem3} and \ref{Lem4}, the equality actually holds. However, we only need the inequality in the direction we stated, and this simplifies our proofs. These two final lemmas are glued together in the proof of theorem \ref{TeoCod} by a technical result in $\OG$. 

\begin{lemma}\label{LemSup}
Let us consider for each $i\in\mathbb{N}$ a finite set $A_i\subset \mathbb{O}$. Then,
	$$\sup\{\sup{A_i};\,i\in \mathbb{N}\}=\sup\{\cup_{i\in \mathbb{N}}A_i\}$$
\end{lemma}

Before we prove the lemmas, we would like assume them to be true and conclude the proof of Theorem \ref{TeoCod}. 

\begin{proof}[Proof of theorem \ref{TeoCod}] 
For any $\F\in \Sigma$, we know by lemma \ref{Lem1} that  $[c_{f,\F}(n)]\leq [s_{f,\e}(n)]$ $\leq o(f)$. When the supremum is taken over $\F\in \Sigma$, we deduce that 
\[o(f)\geq \sup\{[c_{f,\F}(n)]\in \OG: \F\in \Sigma\}.\]

For the other inequality, let us take $\e>0$. By lemma \ref{Lem2}, we obtain a finite family $\F$ of wandering compact neighborhoods such that $[s_{f,\e}(n)]\leq[ c_{f,\F}(n)]$. We apply lemma \ref{Lem3} to said family, and we see that
\[ [s_{f,\e}(n)] \leq \sup\{[c_{f,\F'}(n)]; \F'\subset \F \,\text{is disjoint}\}.\]
For each $\F'\subset \F$ formed by disjoint sets, we apply lemma \ref{Lem4} and lemma \ref{LemSup} to conclude
\[ [s_{f,\e}(n)] \leq \sup\{[c_{f,\F'}(n)]; \F'\subset \F \text{ and }\F'\in \Sigma\}.\]
Therefore,
\[[s_{f,\e}(n)] \leq \sup\{[c_{f,\F}(n)]\in \OG: \F\in \Sigma\};\]
as this happens for every $\e$, when we take the supremum over $\e$, we infer our final inequality. The switch of $\Sigma$ by $\Sigma_\delta$ is possible because in lemma \ref{Lem2}, we can select a family $\F$ in which every element has a diameter less than $\delta$. This concludes the proof of theorem \ref{TeoCod}. 
\end{proof}

Having proven theorem \ref{TeoCod}, we now move to proving the previously stated lemmas. We would like to highlight that the proof of lemma \ref{Lem1} is contained in the proof of Lemma 2.4 in \cite{HaRo19}. As the proof of our result uses only a fraction of their proof, we decided to include it in this manuscript, although the adaptations are minor. In the proof of lemma \ref{Lem2}, we find the technical jump from a non-wandering set with only one point to finite non-wandering set. The idea to solve this is completely new. For the proof of lemmas \ref{Lem3} and \ref{Lem4}, something similar happens to lemma \ref{Lem1}, yet in this case, the technicalities of working in $\OG$ are deeper; therefore, we feel compelled to include said proofs. 

\begin{proof}[Proof of Lemma \ref{Lem1}]

	Let $\F=\{Y_1,\cdots ,Y_L\}$ be a disjoint family of the subsets of $\F$ and choose some compact disjoint wandering neighborhoods $U_1,\cdots, U_L$ of the elements $Y_1,\cdots, Y_L$, respectively. Let $\e>0$ be smaller than the distance from $Y_i$ to the complement of $U_i$ for every $i=1,\cdots, L$. 
	
	Fix some positive integer $n$. For every $\mathcal{G}\subset \F$, let $\A_n(\F,\mathcal{G})$ denote a set of elements of $\A_n(f,\F)$ whose set of letters is exactly $\mathcal{G}\cup \{\infty_{\F}\}$. We fix some $\mathcal{G}\subset \F$ and we consider two points $x,y$ in $M$ and two words $\underline{w}=(w_0,\cdots, w_{n-1})$, $\underline{z}=(z_0,\cdots, z_{n-1})$ in $\A_n(\F,\mathcal{G})$, which represent the orbits $(x,\cdots, f^{n-1}(x))$ and $(y,\cdots, f^{n-1}(y))$, respectively. Then, we claim that if the symbols $\underline{w}$ and $\underline{z}$ are distinct, points $x$ and $y$ are $(n,\e)$-separated. Indeed, let $i\in\{0,\cdots, n-1\}$ be such that $w_i\neq z_i$. If both $w_i\neq \infty$, $z_i\neq \infty$, then $f^i(x)$ and $f^i(y)$ belong to the distinct sets of $Y_i$, and they are more than $\e$ apart. If, say, $w_i=\infty$, then $f^i(y)\in Y_{z_i}$ and $f^i(x)\notin Y_{z_i}$. By definition of $\A_n(\F,\mathcal{G})$, there exists some $j\neq i$ in $\{0,\cdots, n-1\}$ such that $f^j(x)\in Y_{z_i}\subset U_{z_i}$. As $U_{z_i}$ is wandering, we see that $f^i(x)\notin U_{z_i}$; thus, yet again, $f^i(x)$ and $f^i(y)$ are more than $\e$ apart, and the claim is proved. 
	
	As an immediate consequence, we deduce $\#\A_n(\F,\mathcal{G})=c_{f,\mathcal{G}}(n)\leq s_{f,\e}(n)$. Since the sets $\A_n(\F,\mathcal{G})$ form a partition of $\A_n(f,\F)$ into $2^{\# \F}=2^L$ elements, we infer	
	$$c_{f,\F}(n)\leq \sum_{\mathcal{G}\subset \F} c_{f,\mathcal{G}}(n)\leq 2^L s_{f,\e}(n),$$
	as we wanted.
\end{proof}

\begin{proof}[Proof of lemma \ref{Lem2}]

Let us suppose that $\Omega(f)=\{p_1,\cdots, p_k\}$. We want to show that for every $\e >0$, there exists a finite family $\F$ of wandering compact neighborhoods of $M\setminus\Omega(f)$ and a constant $B>0$ such that $s_{f,\e}(n)\leq B c_{f,\F}(n)$, for every $n$. Given $\e >0$, let $\F=\{Y_1,\cdots, Y_L\}$ be a family of wandering subsets of $X\setminus\Omega(f)$ with diameters less than $\min\{\e,\delta\}$ and such that each connected component of $M\setminus \cup\F$ contains a point $p_i \in \Omega(f)$ and has a diameter of less than $\e$. Let us denote such components as $V_1, V_2, \cdots, V_k$. If necessary, we could reconstruct the family to ensure the property, if $Y_i\cap Y_j\neq \emptyset$, then $Y_i\cup Y_j$ is wandering. This is not required in this lemma, but it will be in the next one. 
	
	We choose a positive integer $n$ and consider a maximal $(n,\e)$-separated set $E_n$. Let $\Phi:E_n \rightarrow \A_{n}(f,\F)$ be the map that associates for every point $x$ in $E_n$ some coding $\Phi(x)=w \in \A_{n}(f,\F)$ of the sequence $(x,f(x),\cdots, f^{n-1}(x))$ with respect to the family $\F$. Although the map $\Phi$ may not be injective, we will show that the cardinality of set $\Phi^{-1}(w)$ is bounded by a constant that does not depend on $n$, for every word $w\in \A_{n}(f,\F)$.

	To prove the previous assertion, we construct an auxiliary graph $G$. Its vertices are given by the set $V(G)=\{Y_1,\cdots, Y_L, V_1,\cdots, V_k\}$, and the edges are the set $E(G)$ of pairs $(b_1,b_2)\in V(G)^2$ such that $f(b_1)\cap b_2\neq \emptyset$. 
	
	Note that $G$ satisfies the following properties:
	\begin{enumerate}
	\item  There is no edge of type $(Y_i,Y_i)$, as $Y_i$ is a wandering set for every $i=1,\cdots, L$. More generally, there is no path in graph $G$ with both initial and final vertices $Y_i$.
	\item If $\e$ is small enough and there exists an edge of type $(Y_i,V_j)$, then there is no edge of type $(Y_i,V_l)$ with $l\neq j$. To see this, let $ d=\min \{\frac{d(p_i,p_j)}{2},i\neq j\}$. By the uniform continuity of $f$, we know that there exists $\delta(d) > 0$, such that $d(x,y)<\delta(d)$ implies $d(f(x),f(y))<d$, for every $x,y \in M$. If we choose $\e<\delta(d)$, then for every $x_j,x_l \in Y_i$, we see $d(f(x_j),f(x_l))<d$.
	
	\item If $\e$ is small enough and there exists an edge of type $(V_i,V_j)$, then there is no edge of type $(V_i,V_l)$ with $l\neq j$. That is, in the edges of type $(V_i, V_j)$, each $V_i$ is in only one edge as the initial vertex and in only one edge as the final vertex. 
	\end{enumerate}
	
	Recall that a path in graph $G$ is a finite sequence of edges of the form $$[(b_0,b_1),(b_1,b_2),\cdots, (b_{j-1},b_{j})],$$ where $b_i\in V(G)$. The number of edges in a path is called its length. A path in which all the edges are distinct is a trail. We shall simplify the notation by describing an edge in $E(G)$ as $a_i$. 
		
Let us consider $P_n(G)$ as the set of all paths in $G$ with length $n-1$ and that the map $P:E_n \rightarrow P_n(G)$ is defined as follows: $P$ associates a path $P(x)=[(b_0,b_1),(b_1,b_2),\cdots,$  $(b_{n-2},b_{n-1})]$ to every point $x\in E_n$ such that $f^i(x) \in b_i$ for $i=0,\cdots, n-1$.
	
We claim that the map $P$ is injective. In fact, consider $x, y \in E_n$ such that $P(x)=P(y)$. As the diameter of the sets $b_j\in V(G)$ is less than $\e$, we see that $d(f^i(x),f^i(y))$ $ \leq \e$ for all $i\leq n-1$, and therefore, $x=y$.

	Now, consider the map $Q:P_n(G)\rightarrow \A_n(f,\F)$ defined as follows: For each path $p=[(b_0,b_1),(b_1,b_2),\cdots, (b_{n-2},b_{n-1})]\in P_n(G)$, where $b_i\in V(G)$, the map $Q$ associates a word $Q(p)=w=(b^{*}_{0}, b^{*}_{1}, \cdots, b^{*}_{n-1}) \in \A_n(f,\F)$ and each $b^{*}_{i}$ is given by
	
	$$b_{i}^{*}=\left\{\begin{matrix}
	Y_{j_i},\text{ if }\, b_i=Y_{j_i} \\ 
	\infty, \text{ if }\, b_i=V_{j_i} 
	\end{matrix}\right. .$$
	
Observe that we can choose the map $P$ such that $Q\circ P = \Phi$. 

\begin{center}
	\begin{tikzpicture}[x=0.75pt,y=0.75pt,yscale=-1,xscale=1]
		\draw    (167.5,90) -- (207.5,90) ;
		\draw [shift={(209.5,90)}, rotate = 180] [color={rgb, 255:red, 0; green, 0; blue, 0 }  ][line width=0.75]    (10.93,-3.29) .. controls (6.95,-1.4) and (3.31,-0.3) .. (0,0) .. controls (3.31,0.3) and (6.95,1.4) .. (10.93,3.29)   ;
		\draw    (199.5,50) -- (227.91,68.96) ;
		\draw [shift={(229.57,70.07)}, rotate = 213.72] [color={rgb, 255:red, 0; green, 0; blue, 0 }  ][line width=0.75]    (10.93,-3.29) .. controls (6.95,-1.4) and (3.31,-0.3) .. (0,0) .. controls (3.31,0.3) and (6.95,1.4) .. (10.93,3.29)   ;
		\draw    (170.5,51) -- (143.66,68.89) ;
		\draw [shift={(142,70)}, rotate = 326.31] [color={rgb, 255:red, 0; green, 0; blue, 0 }  ][line width=0.75]    (10.93,-3.29) .. controls (6.95,-1.4) and (3.31,-0.3) .. (0,0) .. controls (3.31,0.3) and (6.95,1.4) .. (10.93,3.29)   ;
		
		\draw (175,35.4) node [anchor=north west][inner sep=0.75pt]    {$E_n$};
		\draw (115,78.4) node [anchor=north west][inner sep=0.75pt]    {$P_{n}( G)$};
		\draw (217,77.4) node [anchor=north west][inner sep=0.75pt]    {$\mathcal{A}_{n}( f,\mathcal{F})$};
		\draw (145,37.4) node [anchor=north west][inner sep=0.75pt]    {$P$};
		\draw (176,95.4) node [anchor=north west][inner sep=0.75pt]    {$Q$};
		\draw (219,36.4) node [anchor=north west][inner sep=0.75pt]    {$\Phi $};
	\end{tikzpicture}
\end{center}	

Let us consider $T(G)$ as the set of all the trails in $G$. Given a path $p=[a_0,a_1,\cdots, a_{n-2}]\in P_n(G)$, we select from $p$ the edges $a_{k_i}$ by induction as follows: $k_0=0$ and $k_i=\min \{k>k_{i-1}; a_k\neq a_{k_{j}}, \text{ for } j\leq i-1\}$. With this, we construct a new path $[a_{k_0},a_{k_1},\cdots, a_{k_l}]$ for some $l>0$. This process eliminates the repeated edges of $p$, transforming the path $p$ with $n-1$ edges into a trail in $T(G)$ with the same edges. Therefore, we construct a map $\phi:P_n(G)\to T(G)$, which removes repeated edges. In particular, for our case, by properties 1, 2 and 3, the map $\phi$ eliminates only the edges of type $(V_i, V_j)$.
	
The following diagram adds the new map 	$\phi: P_n(G)\rightarrow T(G)$ to the previous one.

\begin{center}
	\begin{tikzpicture}[x=0.75pt,y=0.75pt,yscale=-1,xscale=1]
		
		\draw    (140,106) -- (139.6,133.07) ;
		\draw [shift={(139.57,135.07)}, rotate = 270.84000000000003] [color={rgb, 255:red, 0; green, 0; blue, 0 }  ][line width=0.75]    (10.93,-3.29) .. controls (6.95,-1.4) and (3.31,-0.3) .. (0,0) .. controls (3.31,0.3) and (6.95,1.4) .. (10.93,3.29)   ;
		\draw    (167.5,90) -- (207.5,90) ;
		\draw [shift={(209.5,90)}, rotate = 180] [color={rgb, 255:red, 0; green, 0; blue, 0 }  ][line width=0.75]    (10.93,-3.29) .. controls (6.95,-1.4) and (3.31,-0.3) .. (0,0) .. controls (3.31,0.3) and (6.95,1.4) .. (10.93,3.29)   ;
		\draw    (199.5,50) -- (227.91,68.96) ;
		\draw [shift={(229.57,70.07)}, rotate = 213.72] [color={rgb, 255:red, 0; green, 0; blue, 0 }  ][line width=0.75]    (10.93,-3.29) .. controls (6.95,-1.4) and (3.31,-0.3) .. (0,0) .. controls (3.31,0.3) and (6.95,1.4) .. (10.93,3.29)   ;
		\draw    (170.5,51) -- (143.66,68.89) ;
		\draw [shift={(142,70)}, rotate = 326.31] [color={rgb, 255:red, 0; green, 0; blue, 0 }  ][line width=0.75]    (10.93,-3.29) .. controls (6.95,-1.4) and (3.31,-0.3) .. (0,0) .. controls (3.31,0.3) and (6.95,1.4) .. (10.93,3.29)   ;
		
		\draw (175,35.4) node [anchor=north west][inner sep=0.75pt]    {$E_n$};
		\draw (115,78.4) node [anchor=north west][inner sep=0.75pt]    {$P_{n}( G)$};
		\draw (217,77.4) node [anchor=north west][inner sep=0.75pt]    {$\mathcal{A}_{n}( f,\mathcal{F})$};
		\draw (123,139.4) node [anchor=north west][inner sep=0.75pt]    {$T( G)$};
		\draw (120,107.4) node [anchor=north west][inner sep=0.75pt]    {$\phi $};
		\draw (145,37.4) node [anchor=north west][inner sep=0.75pt]    {$P$};
		\draw (176,95.4) node [anchor=north west][inner sep=0.75pt]    {$Q$};
		\draw (219,36.4) node [anchor=north west][inner sep=0.75pt]    {$\Phi $};
	\end{tikzpicture}
\end{center}	

We assert now that for each word $w \in \A_n(f,\F)$, the restriction $\phi:Q^{-1}(w)\rightarrow T(G)$ of the map $\phi$ is injective. Let us assume the assertion for now and observe that for every word $w \in \A_n(f,\F)$:
$$ \# (\Phi^{-1}(w))= \# (P^{-1}\circ Q^{-1}(w)) \leq \# (Q^{-1}(w))=\# (\phi \circ Q^{-1}(w))\leq \# (T(G)).$$

As the graph $(G,V(G))$ is finite and $\# (T(G))$ is a constant that does not depend on $n$, let us denote it by $\# (T(G))=B$. Thus,
	$$s_{f,\e}(n) = \# E_n \leq \max\{\# \Phi^{-1}(w):w\in \A_n(f,\F)\} \# \A_n(f,\F) \leq B \# (\A_n(f,\F)).$$
	
	Then,
	$$s_{f,\e}(n)\leq B c_{f,\F}(n).$$

Now, the assertion needs to be proved that for each word $w \in \A_n(f,\F)$, the restriction $\phi:Q^{-1}(w)\rightarrow T(G)$ of the map $\phi$ is injective.

Let $p_1$ and $p_2$ be two paths in $P_n(G)$ such that $Q(p_1)=Q(p_2)=w \in \A_n(f,\F)$ and $\phi(p_1)=\phi(p_2)\in T(G)$. We want to show that $p_1=p_2$. If we write
	$$\begin{matrix}
		p_1= &[(b^{1}_{0},b^{1}_{1}), (b^{1}_{1},b^{1}_{2}), \cdots, (b^{1}_{n-2},b^{1}_{n-1})] \\ 
		p_2= &[(b^{2}_{0},b^{2}_{1}), (b^{2}_{1},b^{2}_{2}), \cdots, (b^{2}_{n-2},b^{2}_{n-1})] 
	\end{matrix},$$
	let $(b^{1}_{i},b^{1}_{i+1})$ for $i\in\{0,\cdots, n-2\}$ be an edge of the path $p_1$. Then, we have three possibilities.
	\begin{itemize}
	\item[1.] $(b^{1}_{i},b^{1}_{i+1})$ is an edge of type $(Y_{j_i},Y_{j_{i+1}})$:
	
	Since $Q(p_1)=Q(p_2)=w=(b^{*}_{0},b^{*}_{1}, \cdots, b^{*}_{n-1})$, we know that $b^{1}_{i}=Y_{j_i}$ and $b^{1}_{i+1}=Y_{j_{i+1}}$. Thus, $(b^{2}_{i},b^{2}_{i+1})=(Y_{j_i},Y_{j_{i+1}})=(b^{1}_{i},b^{1}_{i+1})$.
	
	\item[2.] $(b^{1}_{i},b^{1}_{i+1})$ is an edge of type $(Y_{j_i},V_{j_{i+1}})$:
	
	This type of edge does not repeat and is, therefore, not eliminated by the map $\phi$. Thus, we infer that $a_l=(Y_{j_i},Y_{j_{i+1}})$, for some $l\leq i$, is an edge of $\phi(p_1)=\phi(p_2)$, and therefore, $(b^{2}_{i},b^{2}_{i+1})=a_l=(Y_{j_i},V_{j_{i+1}})=(b^{1}_{i},b^{1}_{i+1})$.
	
	\item[3.] $(b^{1}_{i},b^{1}_{i+1})$ is an edge of type $(V_{j_i},V_{j_{i+1}})$:
	
	\begin{itemize}
	\item If this edge is not eliminated by the map $\phi$, then we know that $a_l=(V_{j_i},V_{j_{i+1}})$, for some $l\leq i$, is an edge of $\phi(p_1)=\phi(p_2)$, and thus, $(b^{2}_{i},b^{2}_{i+1})=a_l=(V_{j_i},V_{j_{i+1}})=(b^{1}_{i},b^{1}_{i+1})$.
	
	\item If this edge is eliminated by the map $\phi$, then we consider the previous edge 	
	$$\dy (b^{1}_{i-1},b^{1}_{i})=\left\{\begin{matrix}
		(Y_{j_{i-1}},V_{j_i})\\ 
		(V_{j_{i-1}},V_{j_i})
	\end{matrix}\right. :$$
	\begin{itemize}
		\item[--] If $(b^{1}_{i-1},b^{1}_{i})=(Y_{j_{i-1}},V_{j_i})$, then the edge $(b^{1}_{i},b^{1}_{i+1})$ will not be eliminated by the map $\phi$, and we deduce that $a_l=(V_{j_i},V_{j_{i+1}})$, for some $l\leq i$, is an edge of $\phi(p_1)=\phi(p_2)$. Therefore, $(b^{2}_{i},b^{2}_{i+1})=a_l=(V_{j_i},V_{j_{i+1}})=(b^{1}_{i}, b^{1}_{i+1})$.
		
		\item[--] If $(b^{1}_{i-1},b^{1}_{i})=(V_{j_{i-1}},V_{j_i})$, then it may or may not be eliminated by the map $\phi$. If it is eliminated, then we repeat the argument until we obtain $(b^{1}_{i-m},b^{1}_{i-m+1})$ for some $m\leq i$, which is not eliminated by the map $\phi$. This must happen because the paths have a finite length.	
	\end{itemize}
	\end{itemize}
\end{itemize}
	
In all cases, we conclude that the edges of the paths $p_1$ and $p_2$ are necessarily the same, as we wanted.
\end{proof}

Before we prove lemma \ref{Lem3}, we need  to prove lemma \ref{LemSup} and two other intermediary results.

\begin{proof}[Proof of lemma \ref{LemSup}]
Suppose that $A_i=\{[a_{i1}(n)], \cdots , [a_{ik_i}(n)]\}$ for every $i\in \mathbb{N}$. It is easy to see that $\sup\{\cup_{i\in\mathbb{N}}A_i\}\geq [a_{ij}(n)]$, for all $[a_{ij}(n)]\in A_i$, for all $i\in\mathbb{N}$. Then, $\sup A_i\leq \sup\{\cup_{i\in\mathbb{N}}A_i\}$, for all $i\in\mathbb{N}$, and $$\sup\{\sup A_i;\,i\in\mathbb{N}\}\leq \sup\{\cup_{i\in\mathbb{N}}A_i\}.$$
	 
	On the other hand, for all $[a_{ij}(n)]$, we know $[a_{ij}(n)] \leq \sup A_i\leq \sup\{\sup A_i;\,i\in\mathbb{N}\}$, and $$\displaystyle \sup\{\cup_{i\in \mathbb{N}}A_i\}\leq \sup\{\sup{A_i};\,i\in \mathbb{N}\}.$$
\end{proof}

\begin{lemma}\label{LemMax}
	Let $[a(n)],[b(n)] \in \mathbb{O}$, then $$\sup\{[a(n)],[b(n)]\}=[\max\{a(n),b(n)\}]\in \mathbb{O}.$$
\end{lemma}

\begin{proof}
	We know that $[a(n)]\leq [\max\{a(n),b(n)\}] $ and $[b(n)]\leq [\max\{a(n),b(n)\}]$. Then $\sup\{[a(n)],[b(n)]\}\leq [\max\{a(n),b(n)\}]$. However, if we suppose $$\sup\{[a(n)],[b(n)]\}< [\max\{a(n),b(n)\}],$$ then there exists $[c(n)]\in \mathbb{O}$ such that $$\sup\{[a(n)],[b(n)]\}<[c(n)]< [\max\{a(n),b(n)\}].$$ 
	As $\sup\{[a(n)],[b(n)]\}<[c(n)]$, and by definition, $[a(n)]\leq\sup\{[a(n)],[b(n)]\},$ then $[a(n)]\leq [c(n)]$. Using the same argument, we also deduce $[b(n)]\leq [c(n)]$. Thus, there must exist constants $k_1,k_2>0$ such that $a(n)\leq k_1 c(n)$ and $b(n)\leq k_2 c(n)$. Let $k=\max\{k_1,k_2\}$ and observe that $\max\{a(n),b(n)\}\leq kc(n)$. This implies $[\max\{a(n),b(n)\}]$ $\leq [c(n)]$, a contradiction. Therefore, $$\sup\{[a(n)],[b(n)]\}=[\max\{a(n),b(n)\}],$$ as we wanted. 
\end{proof}	

We would like to point out that $\sup\{[a(n)],[b(n)]\}$ may not be either $[a(n)]$ or $[b(n)]$ but, in fact, a third order of growth. It is only one of them when they are comparable according to the partial order. The previous proposition happens to only be true when we consider the supremum of a finite set of orders of growth. Otherwise, we could have that the supremum is an element of $\overline{\mathbb{O}}\setminus\mathbb{O}$.

\begin{lemma}\label{LemWanAdd}
Let $\F$ be a finite family of subsets of $M$. If $\F=\{Y_1,\cdots, Y_L\}$ is such that $Y_1\cup Y_2$ is wandering, then 
$$[c_{f,\F}(n)]=\sup\{[c_{f,\F_1}(n)],[c_{f,\F_2}(n)]\},$$
where $\F_1=\{Y_1,Y_3,\cdots, Y_L\}$ and $\F_2=\{Y_2,Y_3,\cdots, Y_L\}$.
\end{lemma}

\begin{proof}
As $Y_1\cup Y_2$ is wandering, no word in $\A_n(f,\F)$ contains both letters $'Y_1'$ and $'Y_2'$; as a consequence,
$$\A_n(f,\F)=\A_n(f,\F_1)\cup \A_n(f,\F_2).$$
This is a disjoint union. Thus,
$$\begin{array}{rcl}
c_{f,\F}(n)&=& c_{f,\F_1}(n) +c_{f,\F_2}(n)\\
&\leq & 2 \max \{c_{f,\F_1}(n),c_{f,\F_2}(n)\},
\end{array}$$
and then by lemma \ref{LemMax},
$$[c_{f,\F}(n)]\leq [\max \{c_{f,\F_1}(n),c_{f,\F_2}(n)\}]=\sup\{[c_{f,\F_1}(n)],[c_{f,\F_2}(n)]\}.$$
To reverse the inequality, it can be easily observed that $c_{f,\F_i}(n)\leq c_{f,\F}(n)$ for $i=1,2$. Therefore, $$\max \{c_{f,\F_1}(n),c_{f,\F_2}(n)\}\leq c_{f,\F}(n) ,$$
and by lemma \ref{LemMax}, 
$$\sup\{[c_{f,\F_1}(n)],[c_{f,\F_2}(n)]\}=[\max \{c_{f,\F_1}(n),c_{f,\F_2}(n)\}] \leq [c_{f,\F}(n)],$$
which entails the wanted equality.	
\end{proof}

\begin{proof}[Proof of lemma \ref{Lem3}]

Let us consider $\F=\{Y_1,\cdots,Y_L\}$ a finite family of wandering compact neighborhoods of $M\setminus \Omega(f)$ such that if $Y_i$ meets $Y_j$, then $Y_i\cup Y_j$ is a wandering set. 

If $\F$ is disjoint, $[c_{f,\F}(n)]\in\{[c_{f,\F'}(n)]; \F'\subset \F \text{ is disjoint}\}$, and then, $[c_{f,\F}(n)]\leq \sup\{[c_{f,\F'}(n)]; \F'\subset \F \text{ is disjoint}\}$.
	
Otherwise, we select two distinct elements that intersect. To simplify the notation, we can name them $Y_1$ and $Y_2$. Then, by lemma \ref{LemWanAdd}, we know that
$$[c_{f,\F}(n)]=\sup\{[c_{f,\F_1}(n)],[c_{f,\F_2}(n)]\},$$
where $\F_1=\{Y_1,Y_3,\cdots, Y_k\}$ and $\F_2=\{Y_2,Y_3,\cdots, Y_k\}$. If $\F_1$ and $\F_2$ are both disjoint, we see 
$$\sup\{[c_{f,\F_1}(n)],[c_{f,\F_2}(n)]\}\leq \sup\{[c_{f,\F'}(n)]; \F'\subset \F \text{ is disjoint}\},$$
and we conclude.
If either $\F_1$ or $\F_2$ is not a disjoint family, we apply the previous reasoning to each. Let us suppose that both of them are not disjoint families. Based on the previous argument, we construct four families $\F_{11}$, $\F_{12}$, $\F_{21}$ and $\F_{22}$ such that
$$[c_{f,\F}(n)]\leq \sup\{\sup\{[c_{f,\F_{11}}(n)],[c_{f,\F_{12}}(n)]\},\sup\{[c_{f,\F_{21}}(n)],[c_{f,\F_{22}}(n)]\}\}.$$
From lemma \ref{LemSup}, we deduce that 
$$[c_{f,\F}(n)]\leq \sup\{[c_{f,\F_{11}}(n)],[c_{f,\F_{12}}(n)],[c_{f,\F_{21}}(n)],[c_{f,\F_{22}}(n)]\}.$$

We can repeat this process until every family in the $\sup$ on the right side of the inequality is disjoint, and therefore,
$$[c_{f,\F}(n)]\leq \sup\{[c_{f,\F'}(n)]; \F'\subset \F \text{ is disjoint}\}.$$
\end{proof}

\begin{proof}[Proof of Lemma \ref{Lem4}]
Let us consider $\F$ a finite family of disjoint wandering compact neighborhoods that is not mutually singular. We will show that $[c_{f,\F}(n)]$ can be calculated in terms of its subfamilies $\F'\subset \F$ that are mutually singular. The lemma follows by a finite backward induction, as only one set is always mutually singular.
	
	We assume that $\F=\{Y_1,\cdots,Y_L\}$ is not mutually singular. Like in the proof of lemma \ref{Lem1}, for every $\F'\subset \F$, we denote $\A_n(\F,\F')$ as the set of elements of $\A_n(f,\F)$ whose set of letters is exactly $\F'\cup \{\infty\}$. In particular, the elements of $\A_n(\F,\F )$ use the same letters $Y_{1},\cdots, Y_{L}$ and $\infty$. As the $Y_i$'s are wandering, each letter except $\infty$ appears at most once; thus,
	$$\A_{n}(f,\F)=\bigcup_{\F' \subset \F} \A_n (\F,\F').$$
	
	As the $Y_i$'s are not mutually singular, there exists a number $M$ such that if a point $x\in X$ satisfies $f^{n_i}(x)\in Y_i$ and $f^{n_j}(x)\in Y_j$, then $|n_i-n_j|\leq M$. For every $i\neq j \in \{1, \cdots, L\}$, denote the set of elements $w$ of $\A_n(\F,\F)$ by $\A_n(\{i,j\})$ such that the letters $Y_i$ and $Y_j$ appear at places at most $M$ apart. We know
	
	$$\A_n(\F,\F)=\bigcup_{(i,j)} \A_n (\{i,j\}),$$
	and
	$$\A_{n}(f,\F)=\bigcup_{\F' \subsetneq \F} \A_n (\F,\F') \cup \bigcup_{(i,j)} \A_n (\{i,j\}).$$
	Then,
	$$c_{f,\F}(n)=\sum_{\F' \subsetneq \F}c_{\F,\F'}(n) + \sum_{(i,j)}c_{\{i,j\}}(n).$$
	Thus,
	$$\begin{array}{rcl}
		[c_{f,\F}(n)]&=&\left[\sum_{\F' \subsetneq \F}c_{\F,\F'}(n) + \sum_{(i,j)}c_{\{i,j\}}(n)\right] \\ \\
		&\leq&  \left[\max\{\sum_{\F' \subsetneq \F}c_{\F,\F'}(n),\sum_{(i,j)}c_{\{i,j\}}(n)\}\right] \\ \\
		&=& \sup\{[\sum_{\F' \subsetneq \F} c_{\F,\F'}(n)],[\sum_{(i,j)}c_{\{i,j\}}(n)]\} \\ \\
		&\leq& \sup\{\sup_{\F' \subsetneq \F}\{ [c_{\F,\F'}(n)]\}, \sup_{(i,j)}\{[c_{\{i,j\}}(n)]\}\}\\ \\
		&=& \sup\{[c_{\F,\F'}(n)],[c_{\{i,j\}}(n)]; \F' \subsetneq \F, (i,j) \}.
	\end{array}$$

	Moreover, for each positive integer $n$, if $w$ is an element of $\A_n(\{i,j\})$, let $w'$ be obtained from $w$ by changing the letter \textquoteleft$Y_i$', which appears exactly once in $w$, to \textquoteleft$\infty$', and $w'$ is uniquely determined. Since $\F$ is a disjoint family, we infer $w' \in \A_n(\F,\F')$, where $\F'=\{Y_1,\cdots,Y_{i-1},Y_{i+1},\cdots,Y_L\}$. The word $w$ also contains the letter \textquoteleft$Y_j$', and the letter \textquoteleft$Y_i$' is at most $M$ places apart. Thus, $w'$ has at most $2M$ inverse images under the map $w\mapsto w'$. Then, we deduce that
	$$\# \A_n(\{i,j\})\leq 2M \# \A_n(\F,\F').$$
Thus, 
	$$[c_{\{i,j\}}(n)]\leq [c_{\F,\F'}(n)],$$
and therefore, 
	$$[c_{f,\F}(n)]\leq \sup\{[c_{\F,\F'}(n)]; \F' \subsetneq \F\}.$$
Since each $[c_{\F,\F'}(n)]\leq [c_{f,\F'}(n)]$, we see that 
\[ [c_{f,\F}(n)]\leq \sup\{[c_{f,\F'}(n)]; \F' \subsetneq \F\}.\]
As explained in the beginning of the proof by a backward induction and lemma \ref{LemSup}, we deduce an inequality, where the left side has 
$[c_{f,\F}(n)]$ and the right side has only families that are mutually singular. This concludes the lemma.   
\end{proof}


\begin{thebibliography}{99}

\bibitem{ArCaMo17} A. Artigue, D. Carrasco-Olivera, and I. Monteverde, Polynomial Entropy and Expansivity, Acta Math. Hungar. 152(1) (2017), 152-140.

\bibitem{LaBe14} P. Bernard and C. Labrousse, An entropic characterization of the flat metrics on the two torus, Geom. Dedicata 180(1) (2014), 187-201.

\bibitem{CaGo21} M. J. D. Carneiro and J. B. Gomes, Polynomial entropy for interval maps and lap number,	Qual. Theory Dyn. Syst. 20 (2021), 1-16.

\bibitem{CaLyMa05} A. De Carvalho, M. Lyubich, and M. Martens, Renormalization in the Hénon Family, I: Universality But Non-Rigidity. J. Stat. Phys. 121(5-6) (2005), 611-669.


\bibitem{CoPu21} J. Correa and E. Pujals, Orders of growth and generalized entropy, J. Inst. Math. Jussieu (2021), 1-33.

\bibitem{CrPu16} S. Crovisier and E. Pujals, Strongly dissipative surface diffeomorphisms, Comment. Math. Helv. 93(2) (2016), 377-400.

\bibitem{CrPuTr20} S. Crovisier, E. Pujals and C. Tresser, Mildly dissipative diffeomorphisms of the disk with zero entropy, Preprint (2020) arXiv:2005.14278.

\bibitem{GaStTr89} J. Gambaudo, S. van Strein, C. and Tresser, Henon-like maps with strange attractors: there exist C$^\infty$ Kupka-Smale diffeomorphisms on S$^2$ with neither sinks nor sources, Nonlinearity 2(2) (1989), 287-304.

\bibitem{HaRo19} L. Hauseux and F. Le Roux, Entropie polynomiale des homéomorphismes de Brouwer, Annales Henri Lebesgue 2, (2019), 39-57.

\bibitem{KaPe19} J. Kati\'c and M. Peri\'c, On the polynomial entropy for Morse gradient systems, Math. Slovaca 69(3) (2019), 611-624.

\bibitem{La13-02} C. Labrousse, Flat metrics are strict local minimizers for the polynomial entropy, Regul. Chaotic Dyn. 17(6) (2012), 479-491.

\bibitem{La12} C. Labrousse, Polynomial growth of the volume of balls for zero-entropy geodesic systems, Nonlinearity 25(11) (2012), 3049-3069.

\bibitem{La13-01} C. Labrousse, Polynomial entropy for the circle homeomorphisms and for $C^1$ nonvanishing vector fields on $\mathbb{T}^2$, Preprint (2013), arXiv:1311.0213.

\bibitem{LaMa14} C. Labrousse and J. P. Marco, Polynomial entropies for Bott integrable Hamiltonian systems, Regul. Chaotic Dyn. 19(3) (2014), 374-414.

\bibitem{Ma13} J. P. Marco, Polynomial entropies and integrable Hamiltonian systems, Regul. Chaotic Dyn. 18(6) (2013), 623-655.

\bibitem{RoRoSn21} S. Roth, Z. Roth and L. Snoha, Rigidity and flexibility of polynomial entropy, Preprint (2021), arXiv:2107.13695.


\end{thebibliography}
\end{document}